\documentclass
[12pt]{article}
\usepackage{chngcntr}
\usepackage{titletoc}
\usepackage[runin]{abstract}
\usepackage[toc,page]{appendix}
\usepackage{multirow}%
\usepackage{xcolor}%
\usepackage{fullpage}
\usepackage{amsmath,amsthm,amsfonts,amssymb}%
\usepackage{enumerate}
\usepackage{natbib}
\usepackage{booktabs}
\usepackage{subcaption}  
\usepackage{graphicx}
\usepackage{tabularx} 
\usepackage{float} 
\usepackage{setspace}
\usepackage[margin=1in]{geometry}
\newtheorem{theorem}{Theorem}
\newtheorem{corollary}[theorem]{Corollary}
\newtheorem{proposition}[theorem]{Proposition}
\newtheorem{lemma}[theorem]{Lemma}

\newtheorem{example}{Example}

\newcommand{\mathds}[1]{#1}
\newcommand{\R}{\mathbb{R}}
\newcommand{\w }[1]{\widehat{#1}}

\renewcommand{\r}{ \rightarrow }

\newcommand{\lr}{ \longrightarrow }

\renewcommand{\t}[1]{\text{#1}}

\DeclareMathOperator{\diam}{diam}

\DeclareMathOperator{\var}{var}

\DeclareMathOperator{\sign}{sign} 
 
\DeclareMathOperator{\argmin}{argmin}

\begin{document}

\title{\bf On an extension of \\the promotion time cure model}

\renewcommand{\baselinestretch}{1.2}

\author{
{\large Fran\c cois P\textsc{ortier} 
\footnote{\small{T\'el\'ecom ParisTech, Universit\'e Paris-Saclay, France. Email: francois.portier@gmail.com}}} \\ {\it T\'el\'ecom ParisTech}
\and
\addtocounter{footnote}{2}
{\large Ingrid V\textsc{{an} K{eilegom}}
\footnote{\small{ORSTAT, KU Leuven, Belgium. Email: ingrid.vankeilegom@kuleuven.be.}} $^{, **}$ } \\ {\it KU Leuven} \vspace*{.3cm}
\and
\addtocounter{footnote}{2}
{\large Anouar E\textsc{l} G\textsc{houch}
\footnote{\small{Institute of Statistics, Biostatistics and Actuarial Sciences, Universit\'e catholique de Louvain, Belgium. Email: anouar.elghouch@uclouvain.be.}}} \\ {\it Universit\'e catholique de Louvain}
}

\date{\today}

\maketitle

\vspace{-1cm}

\noindent
\begin{abstract}
We consider the problem of estimating the distribution of time-to-event data that are subject to censoring and for which the event of interest might never occur, i.e., some subjects are cured. To model this kind of data in the presence of covariates, one of the leading semiparametric models is the promotion time cure model \citep{yakovlev1996}, which adapts the Cox model to the presence of cured subjects. Estimating the conditional distribution results in a complicated constrained optimization problem, and inference is difficult as no closed-formula for the variance is available. We propose a new model, inspired by the Cox model, that leads to a simple estimation procedure and that presents a closed formula for the variance. We derive some asymptotic properties of the estimators and we show the practical behaviour of our procedure by means of simulations. We also apply our model and estimation method to a breast cancer data set.  \\
\end{abstract}


\def\baselinestretch{1.25}
\newpage
\normalsize
\baselineskip 18pt
\setcounter{footnote}{0}
\setcounter{equation}{0}

\section{Introduction}

Survival analysis has been the subject of many statistical studies in the past decades (see e.g.\ \cite{klein+m:2005,grambsch+t:2000}) and is commonly used in clinical trials (see e.g.\ \cite{collett:2015}), where the traditional main goal is to explain the death of patients having a certain disease. When analysing the effect of some covariates $X\in \R^{d}$ on a survival time $T\in \R_{\geq 0}$, a common approach in the literature is based on semiparametric estimation. The seminal paper by \cite{cox1972} introduces the so-called semiparametric \textit{proportional hazards} model, often referred to as the Cox model, which is given by the following set of conditional survival functions defined on $\R_{\geq 0}\times \R^{d}$ :
\begin{align*}
\mathcal P_{1}=\left\{ (t,x) \mapsto  S(t | x) =\exp(-\exp(\gamma^Tx) \Lambda( t ))  \  :\  \gamma\in\R^{d},\ \Lambda\in \mathcal G \right\},
\end{align*}
where $\mathcal G$ is the space of absolutely continuous cumulative hazard functions defined on $\R_{\geq 0}$. In this standard semiparametric model the elements are characterized by the Euclidean parameter $\gamma$, called the regression vector, and the infinite dimensional parameter $\Lambda$, called the cumulative hazard. These parameters are estimated by maximizing the \textit{profile likelihood} for $\gamma$ \citep{cox1972} and by computing the Breslow estimator for $\Lambda$ \citep{breslow1972}. Both estimators are \textit{nonparametric maximum likelihood estimators} (NPMLE), as defined in \cite{murphy1994}. Known asymptotic results include the asymptotic normality \citep{gill1982}, the semiparametric efficiency of the regression parameters \citep{wellner1988} as well as the cumulative hazard \citep{wellner1993,kosorok2008}, and the validity of general bootstrap schemes \citep{wellner1996}.

In many data sets, especially the ones arising from clinical trials, a certain proportion of the individuals will never experience the event of interest. These individuals are referred to as the \textit{cured subjects}. As the survival function $t\mapsto S(t)$ does not tend to $0$ as $t\rightarrow +\infty$ in that case (but rather tends to the proportion of cured subjects), specific models need to be considered to account for the \textit{improperness} of the distribution of $T$. The \textit{promotion time cure model} is an extension of the Cox model specially designed to handle the presence of cured subjects in the data. It is defined as the set of conditional survival functions defined on $\R_{\geq 0} \times \R^{d}$, given by
\begin{align*}
\mathcal P_2=\left\{ (t,x) \mapsto  S(t|x) =\exp\left(-\eta(\beta_1+\beta_2^Tx) F(t)\right) \ :\ \beta=(\beta_1,\beta_2) \in \R^{d +1},\  F \in  \mathcal F \right\},
\end{align*}
where $\mathcal F$ denotes the space of absolutely continuous cumulative distribution functions on $\R_{\geq 0} $ and $\eta:\mathbb R\rightarrow \mathbb R_{> 0} $ is a given function. This model was introduced by \cite{yakovlev1996} and seems appropriate to treat cure data as, for every $x\in \mathbb R^{d}$, $\lim_{t\r +\infty}S(t|x) >0$, so that each subject has a positive chance of being cured. In model $\mathcal P_2$, the parameter vector $\beta$ has an intercept whereas $\gamma$ in model $\mathcal P_1$ does not.  This is because $\lim_{t \rightarrow +\infty} \Lambda(t) = +\infty$ and hence an intercept in model $\mathcal P_1$ would not be identified, whereas in model $\mathcal P_2$ the function $\Lambda(t)$ is replaced by $F(t)$, which tends to 1 as $t \rightarrow +\infty$.  Estimation of $\mathcal P_2$ has been studied by \cite{tsodikov1998a,tsodikov1998b,tsodikov2001,chen1999,ibrahim2001,tsodikov2003,zeng2006,portier+e+v:2017}, among many others. Certain parallels might be drawn between the statistical properties related to the estimators of the classical Cox model and the ones related to the promotion time cure model. The classical estimators of $\beta$ and $F$ are the NPMLE's \citep{zeng2006}. In \cite{zeng2006}, the authors show that the resulting NPMLE is asymptotically normal and moreover that the estimated vector of regression parameters is semiparametrically efficient. In \cite{portier+e+v:2017} it is shown that the whole model is estimated efficiently and the validity of a general weighted bootstrap is proved. 

There is still an important difference between the NPMLE's associated to models $\mathcal P_1$ and $\mathcal P_2$. The NPMLE of the Cox model has a much simpler expression than the NPMLE of the promotion time cure model. Within model $\mathcal P_1$, the estimated regression parameter maximizes a known (explicit) objective function and the estimated cumulative hazard is expressed through a closed formula \citep{gill1982}. Within Model $\mathcal P_2$, the estimated regression parameter is also the maximizer of a certain objective function, but this time the objective function is implicitly defined \citep{portier+e+v:2017}. Moreover, the same is true for the estimated cumulative hazard in $\mathcal P_2$, which is only known up to some quantity implicitly defined. The previous features involve important complications that intervene at two different stages. First, estimators from $\mathcal P_2$ are more difficult to describe, theoretically, than  estimators from $\mathcal P_1$. This eventually deteriorates the accuracy of the confidence intervals or of the testing procedures. Second, the computation of the estimators in $\mathcal P_2$  has some numerical difficulties, e.g., long computation time, problems with local minima, etc. Given this, the question is to know, whether or not, it is legitimate to rely on a complicated estimation procedure for $\mathcal P_2$? In other words, does the presence of cured subjects in the data prevents us from having an estimation procedure as simple as in the Cox model? 

The aim of this paper is to provide a new model dedicated to cure data analysis and for which the NPMLE overpasses the previous difficulties associated with $\mathcal P_2$.

The undesirable complications when estimating $\mathcal P_2$ come from the particular nature of the parameter space $\mathcal F$. This space is formed by cumulative distribution functions $F$ that satisfy the constraint $\lim_{t\rightarrow \infty} F (t) = 1$. Such a constraint is taken into account with the help of a Lagrange procedure involving an additional parameter being implicitly defined, the \textit{Lagrange multiplier}. It turns out that this constraint can be alleviated by including an additional parameter in the model, replacing $F$ by $\theta F$, with $\theta>0$. We define the set of conditional survival functions $\R_{\geq 0} \times \R^{d}\rightarrow \R_{\geq 0} $, given by
\begin{align*}
\mathcal P_3=\big\{  (t,x) \mapsto S(t|x) = \exp(-g(\gamma,x)\theta F(t)) \ :\ (\gamma,\theta)\in \R^q\times \R_{>0} ,\  F \in  \mathcal F \big\},
\end{align*}
where $g:\mathbb R^q \times \mathbb R^{d} \rightarrow \R_{> 0} $ is a given function and $q\in \mathbb N$. Note that in the present form, $\mathcal P_3$ handles biological models as developed in \cite{chen1999} to analyse time to relapse of cancer through the distribution of the carcinogenic cells. It includes also a cure version of the Cox model when $g(\gamma,x)=\exp(\gamma^Tx)$. In this case, it coincides with $\mathcal P_2$ for which $\eta = \exp$. Otherwise $\mathcal P_2$ and $\mathcal P_3$ are different. In $\mathcal P_3$, the role of $\theta$ is interpreted as a simple multiplicative effect on the cumulative distribution, whereas the effect of $\beta_1$ in $\mathcal P_2$ must be analysed depending on the shape of the function $\eta$. 

The main contributions of the paper are listed below. 
\begin{enumerate}[(i)]

\item As the NPMLE of $\mathcal P_3$ is much simpler to evaluate than the one associated to $\mathcal P_2$, the proposed methodology provides a significant improvement in terms of computational ease. In particular, we show that the NPMLE's associated with $\mathcal P_2$ and $\mathcal P_3$ coincide when $\eta=\exp$ and $g(\gamma, x) = \exp(\gamma^ T x)$. Hence our approach provides a new way to compute the NPMLE of $\mathcal P_2$ when $\eta=\exp$ (most commonly used) which is simpler than the existing procedure \cite{zeng2006,portier+e+v:2017}. 
\item We derive the asymptotics of the NPMLE associated with $\mathcal P_3$.  As in the case of the Cox model, we have closed-formulas for the variance of the limiting Gaussian distributions. This allows us to develop some tests and to build confidence intervals on some quantities of interest as for instance the proportion of cure given the value of a covariate $x$. The finite sample size accuracy of the confidence intervals is investigated with the help of simulations.
\item Moreover, as the function $g$ needs to be chosen by the analyst, we consider a likelihood-based methodology to select an appropriate function $g$ among a family of proposals. Such an approach is also followed by \cite{huang:2006}, who investigate spline estimation of the function $g$ in the case of the classical Cox model.
\end{enumerate}

In section \ref{s2} we present the framework of the paper and derive the NPMLE of model $\mathcal P_3$. We also consider the links with the NPMLE of $\mathcal P_2$.
In section \ref{s3}, the asymptotic behaviour of the NPMLE of $\mathcal P_3$ is studied. In sections \ref{s6} and 5, we provide simulations and a real data analysis to give some insights in the finite sample performance of our approach. The proofs are collected in the Appendix.

\section{The data, the model, the estimator}\label{s2}

\subsection{Framework}\label{s21}

We focus on the standard right censoring context : the lifetime $T$ of interest is right censored by some random variable $C$ so that we only observe $Y=\min(T,C)$, $\delta = \mathds 1{\{T\leq C\}}$ and the vector of covariates $X$. This means that we know whether the variable of interest $T$ has been observed or censored. The covariates $X$ are in contrast always observed, and we further denote by $ \mathcal S\subseteq \mathbb R^d$ their support. 
We suppose conditional independence between $T$ and $ C$, given $X$. In practice, as is the case for instance in clinical trials, $C$ might be bounded. This prevents us from observing any cured subjects, defined by $T=+\infty$. A way around this problem is to assume the existence of a threshold $\tau\in \R$ such that 
\begin{align*}
\{T>\tau\}  \Rightarrow  \{T=+\infty\} .
\end{align*}
Therefore whenever $Y$ will be observed to be greater than $\tau$, the individual will be known to be cured. We use model $\mathcal P_3$ for modelling the distribution of $T$ given $X$. Hence we further assume that the conditional survival function of $T$ given $X=x$ is given by $S_0(t|x) = \exp(-g (\gamma_0, x) \theta_0 F_0(t)  )$, for some $\gamma_0\in \mathbb R^q$, $\theta_0\in \mathbb R_{> 0} $, and $F_0$ an absolutely continuous cumulative distribution function. Let $P$ denote the probability measure associated to $(Y,C,X)$. Supposing in addition that $P(C>\tau|X)>0$ a.s., we obtain that $P(Y>\tau |X)>0$ a.s., meaning that every individual can be cured. These assumptions are stated in Section \ref{s3} in (H\ref{cond:identification1}).

A central object in our study is the counting process $N(y)=\delta \mathds 1_{\{Y\leq y \}}$, $y\in \mathbb R _{\geq 0}$, as it possesses some useful martingale properties as developed in \cite{fleming1991} and \cite{andersen+b+g+k:1993}. Define the random process $R(y)=\Delta\mathds 1_{\{Y\geq  y\}}+(1-\Delta)$, $y\in \mathbb R _{\geq 0}$, with $\Delta = \mathds 1_{\{Y\leq \tau \}}$. It equals $1$ whenever the individual is still at risk. The presence of cure implies that $R= 1$ has positive probability. The compensator of $N$ with respect to the $\sigma$-field $\mathcal F_y$ generated by $\{N(u),\, \mathds 1_{\{ Y\leq u,\delta = 0\}},\, X \, :\, 0\leq u\leq y \}$ is the process $y\mapsto \int_0^y g(\gamma_0,X)R(u) \theta_0 dF_0(u)$. That is, $M$ defined by
\begin{align*}
\left\{ \begin{array}{l}
M(0)=0\\
dM(y) = dN(y) -  g(\gamma_0,X)R(y) \theta_0 dF_0(y),\qquad y\in \mathbb R_{\geq 0} 
\end{array} \right.
\end{align*}
is a martingale with respect to $\mathcal F_y$ \cite[Theorem 1.3.1]{fleming1991}. In particular, we have the formula \citep[Theorem 1.5.1]{fleming1991}
\begin{align}\label{formula:martingale}
E\left[ \delta  h(Y,X)  \right] =\int   E\left[  h(u,X)  g(\gamma_0, X) R(u)\right] \theta_0 dF_0(u) 
\end{align}
for any bounded measurable function $h$. Finally, the following identity shall be useful : for any bounded measurable functions $h$ and $\tilde h$, we have  \citep[Theorem 2.4.2]{fleming1991}
\begin{align}\label{formula:quadratic_variation}
E\left[  \int h(u) dM(u) \int \tilde h(u) dM(u) \right] =  \int h(u)\tilde h(u)  E [g(\gamma_0, X) R(u)]  d\Lambda_0(u)   .
\end{align}

\subsection{Nonparametric maximum likelihood}\label{s22}

Let $(T_i,C_i,X_i)_{i\in \mathbb N}$ denote a sequence of independent and identically distributed random variables with law $P$, as described in the previous subsection. The underlying probability measure is denoted by $\mathbb P$. The estimator we consider shall be based on the observed variables : $Y_i=\min(T_i,C_i)$, $\delta_i = \mathds 1{\{T_i\leq C_i\}}$, $X_i$, $i=1,\ldots, n$. Let $N_i(y)=\delta_i \mathds 1_{\{Y_i\leq y \}}$, $R_i(y)=\Delta_i\mathds 1_{\{Y_i\geq y\}}+(1-\Delta_i)$, $\Delta_i = 1_{\{Y_i \le \tau\}}$,  and define the martingale $M_i = N_i- R_i$, for $i=1,\ldots, n$.

Under the current data generating process, assuming that $F$ is absolutely continuous, and assuming non-informative censoring \citep{sasieni:1992}, the likelihood of an observation $(y,\delta,x)$ in model $\mathcal P_3$ is given by
\begin{align}\label{formulalikelihood}
\t{Lik}(y,\delta,x) = \{g(\gamma,x)\theta f(y)\}^\delta\  \exp\Big[ - g(\gamma,x)\theta \{\Delta F(y)+ (1-\Delta)\}\Big],
\end{align}
where $f$ stands for the derivative of $F$.  Model $\mathcal P_3$ can be re-written as the set of all survival functions of the form $ \exp(-g(\gamma,x)\Lambda(t))$ where $\gamma\in \R^{q}$ and $\Lambda $ belongs to $\mathcal G$, the space of absolutely continuous cumulative hazards $\Lambda$ such that $\Lambda(\tau)=\lim_{y\rightarrow +\infty} \Lambda(y)=\theta $. Note that there is a one-to-one relationship between the two sets of parameters $(\theta,F)$ and $\Lambda$, i.e., $ \Lambda =\theta F$ and $\theta=\lim_{t\rightarrow +\infty} \Lambda(t)$. As a consequence the likelihood in (\ref{formulalikelihood}) can be expressed in terms of $(\gamma,\theta,F)$ or equivalently, in terms of $(\gamma, \Lambda)$. Switching from one parametrization to another is straightforward. For the sake of simplicity, we derive the NPMLE with respect to $(\gamma, \Lambda) $ in the next few lines. By following \cite{murphy1994}, the NPMLE is defined as
\begin{align}\label{maximumlikelihood}
 (\w \gamma, \w \Lambda) 
&= \underset{\gamma\in \mathbb R^q , \, \Lambda}{\text{argmax}} 
  \  \sum_{i=1}^n \Big[{\delta_i} \log(g(\gamma,X_i) \Lambda\{Y_i\}) -g(\gamma,X_i) \{\Delta_i \Lambda(Y_i)+ (1-\Delta_i)\Lambda(+\infty) \}\Big], 
\end{align}
the maximum is taken over $\Lambda$ lying in the space of cumulative hazard functions possibly discrete, and $\Lambda\{y\}=\Lambda(y)- \lim_{t\r y^-} \Lambda(t)$ is the size of the jump of $\Lambda$ at $y$. As is common practice for computing the NPMLE in semiparametric models, the above NPMLE might be profiled over the nuisance parameter $\Lambda$ \citep{murphy2000,kosorok2008}. Maximizing along submodels $d\Lambda_s = (1+sh)d\Lambda$, $s\in \mathbb R$, with $h$ a bounded real function, the value of $\Lambda$ which maximizes (\ref{maximumlikelihood}), for each $\gamma\in \mathbb R^q$, is a solution of
\begin{align*}
n^{-1}\sum_{i=1}^n \delta_i h(Y_i) - \int \w Q_\gamma (u)  h(u) d\Lambda(u)=0,
\end{align*}
with $\w Q_\gamma(u)=n^{-1}\sum_{i=1}^n g(\gamma,X_i) R_i(u)$. The solution of the previous equation is given by
\begin{align*}
\w \Lambda_\gamma (y) = n^{-1} \sum_{i=1}^n \frac{\delta_i\mathds 1 _{\{Y_i\leq y \}}}{\w Q_\gamma (Y_i) },\qquad \qquad y\in\mathbb R_{\geq 0}.
\end{align*}
This is then plugged into (\ref{maximumlikelihood}) to get that
\begin{align}
\left\{ \begin{array}{l}
 \displaystyle \w \gamma \in  \underset{\gamma\in \mathbb R^q}{\text{argmax}} 
  \ \prod_{i=1}^n \left \{ {g(\gamma,X_i)} / {\w Q_\gamma (Y_i) }\right\} ^{\delta_i}  \\
 \displaystyle \w \Lambda (y) =n^{-1} \sum_{i=1}^n  \w Q_{\w \gamma} (Y_i) ^{-1} {\delta_i\mathds 1 _{\{Y_i\leq y \}}},\qquad \qquad y\in\mathbb R_{\geq 0}.
\end{array} \right.  \label{NPMLEcoxmodel_gamma_Lambda}
\end{align}
 Back to the parameters $(\theta,F)$ of Model $\mathcal P_3$, the NPMLE is given by 
\begin{align}
  \left\{ \begin{array}{l}   \displaystyle \w\theta = n^{-1} \sum_{i=1}^n  \w Q_{\w \gamma} (Y_i)^{-1} {\delta_i}\\
 \displaystyle \w F(y) = (\w \theta n ) ^{-1} \sum_{i=1}^n  {\w Q_{\w \gamma} (Y_i)^{-1} } {\delta_i\mathds 1 _{\{Y_i\leq y \}}}, \qquad \qquad y\in\mathbb R_{\geq 0}. \end{array}  \right.   \label{NPMLE_coxmodel_theta_F}
\end{align}
At fixed sample size $n$, the quantities involved in the previous equations are well defined as soon as, for instance, there exists $i$ such that $\delta_i=1$ and the maximum in (\ref{NPMLEcoxmodel_gamma_Lambda}) can be taken over a known compact set $B\subset \mathbb R^q$ on which the function $\gamma \mapsto g (\gamma, x) $ is continuous, for every $x\in \mathcal S$.

Note also that the estimation of the parameters depends only on the observed variables $(Y_i,\delta_i,X_i)$ such that $Y_i\leq \tau$, and $(\Delta_i,X_i)$ such that $Y_i>\tau$, $i=1,\ldots, n$. It results that moving the threshold over $[Y_{(n,\delta)},+\infty)$, with $Y_{(n,\delta)}=\max_{i=1,\ldots,n} Y_i \delta_i$, has no effect on the NPMLE. In practice the threshold could then be fixed at $Y_{(n,\delta)}$.

An important point in many situations is to evaluate the proportion of cured subjects in the population under study for a given covariate vector $x\in \mathcal S$, i.e., $p_0(x) = \exp(-g(\gamma_0, x) \theta_0)$. The estimator of $p(x)$, within our framework, naturally follows from the plug-in rule :
\begin{align}\label{def:cure_proportion}
\w p(x) = \exp(-g(\w\gamma,x)\w \theta).
\end{align}

\subsection{Link with other estimators}

\subsubsection{Cox and Breslow estimator}

Model $\mathcal P_3$ is aimed to handle the presence of cured subjects in the data whereas the traditional Cox model, $\mathcal P_1$, is not. However, when $g(\gamma,x) = \exp(\gamma^Tx)$, (\ref{NPMLEcoxmodel_gamma_Lambda}) becomes very close to the well-known formulas of the classical Cox and Breslow estimator of $\gamma$ and $\Lambda$, respectively. As a consequence, the derivation of the asymptotics for $(\w \gamma, \w F,\w \theta)$ is somewhat similar as in the case of the Cox and Breslow estimator, provided for instance in \cite{gill1982}. 
An interesting difference with the Cox and Breslow estimator comes from the fact that
 \begin{align*}
\min_{i = 1,\ldots, n} \w  Q_{\w \gamma} (Y_i) \geq  n^{-1}\sum_{i=1}^n g(\gamma,X_i) (1-\Delta_i).
 \end{align*}
From the framework described in the previous section, we deduce that $E[(1-\Delta)|X] > 0$ and $E[g(\gamma,X)(1-\Delta)]>0$, for every $\gamma\in \mathbb R^q$. Consequently, the decreasing function $u\mapsto  E[g(\gamma,X)R(u)] $ is bounded from below. In Lemma \ref{Lemma:probacv}, see the Appendix, this property is shown to hold for $\w Q_\gamma$, uniformly in $\gamma$, with probability going to $1$. This raises a significant difference with respect to classical Cox estimators in which the quantity corresponding to $\w Q_\gamma$ would go to $0$ at infinity. This in turn implies that the weak convergence of the rescaled $\w \Lambda$ will still hold over $\mathbb R_{\geq 0}$. This is in  contrast with the case of the Cox model for which such a convergence holds on bounded intervals. We refer to \cite{gill1982} for a discussion on the study of the Breslow estimator over $[0,+\infty )$.

\subsubsection{Promotion time cure estimator}
The NPMLE for $\mathcal P_2$ is given by \citep{portier+e+v:2017}, 
\begin{align}
\left\{ \begin{array}{l} \displaystyle
\w \beta \in \underset{(\beta_1,\beta_2)\in \mathbb R^d} {\text{argmax}}\ \prod_{i=1}^n\left\{ \left(\frac{\eta(\beta_1+ \beta_2^TX_i)}{\w Q_{2,\beta}(Y_i)-\w \lambda_\beta}\right)^{\delta_i}  \exp( -\w\lambda_\beta ) \right\} \\
\displaystyle \w G (y) =  n^{-1} \sum_{i=1}^n \frac{\delta_i\mathds 1 _{\{Y_i\leq y \}}}{\w Q_{2,\w \beta} (Y_i) -\w \lambda_{\w \beta} }, \end{array} \right. \label{newopti1}
\end{align}
where $\w Q_{2,\beta}(u)=n^{-1}\sum_{i=1}^n\eta(\beta_1+ \beta_2^TX_i) R_i(u)$ and for every $\beta\in \mathbb R^d$, $\w \lambda_{ \beta}$ is the  smallest number verifying $ \sum_{i=1}^n {\delta_i} / (\w Q_{2,\beta} (Y_i) -\w \lambda_{ \beta} )=n$.
Because the function $\beta \mapsto \w \lambda_\beta$ is implicitly defined, it is more difficult to compute the NPMLE of $\mathcal P_2$ through (\ref{newopti1}), than the one of $\mathcal P_3$ through (\ref{NPMLEcoxmodel_gamma_Lambda}) and (\ref{NPMLE_coxmodel_theta_F}). In particular, solving (\ref{newopti1}) requires to run an optimization procedure over $\beta$ for which, at each iteration, we shall evaluate $\w \lambda_\beta$, by an additional procedure. When $\eta= \exp$, it is actually useless to solve (\ref{newopti1}), since it gives the same results as (\ref{NPMLEcoxmodel_gamma_Lambda}) and (\ref{NPMLE_coxmodel_theta_F}). This is the statement of the following proposition. 

\begin{proposition}\label{proposition:p2_p3}
Suppose that $\eta=\exp$ and $g(\gamma,x)=\exp(\gamma^Tx)$ for every $x\in \mathbb R^{d}$. If there exists $i$ such that $\delta_i=1$, then $\w \beta^T = (\log(\w \theta), \w \gamma^T ) $ and $\w G = \w F $.
\end{proposition}

\section{Asymptotics}\label{s3}

The asymptotic analysis of the NPMLE associated to model $\mathcal P_3$ is inspired from the approach developped for the Cox model in \cite{gill1982}. We may first derive the asymptotic behaviour of the $Z$-estimator $\w\gamma$, and then rely on functional Delta-method type arguments, to describe $\w\Lambda$. The monographs of \cite{vandervaart1996} and \cite{kosorok2008} will be of good help at each of these steps to rely on suitable empirical process techniques. The preliminary study of $\w \gamma$ and $\w \Lambda$ (given in sections \ref{sec:3_1}, \ref{sec:3_2} and \ref{sec:3_3}) will provide the basis to describe the behaviour of $\w p(x)$ defined in (\ref{def:cure_proportion}). 

As it is common for $M$-estimators, the asymptotic study of $\w \gamma$ starts with the establishment of its consistency. In contrast, for $\w\Lambda$,  we will rely on the explicit formula (\ref{NPMLEcoxmodel_gamma_Lambda}) to directly show the weak convergence of $n ^{1/2} (\w \Lambda - \Lambda_0)$. 

\subsection{Consistency of $\w\gamma$}\label{sec:3_1}

The estimator $\w \gamma$ is defined as a maximizer in (\ref{NPMLEcoxmodel_gamma_Lambda}). To obtain the consistency of $\w \gamma$, we classically show that (i) the maximum of the limiting function is well identified and that (ii) the convergence to this limiting function is uniform; see \citet[Theorem 2.1]{newey1994} or \citet[Theorem 5.7]{vandervaart1998}. To obtain the identifiability, we need the following assumptions :

\newcounter{saveenum}
\begin{enumerate}[(\text{H}1)]
\item \label{cond:identification1} The variables $T$ and $C$ are independent given $X$. Moreover, $P(C>\tau|X)>0$ a.s., $P(T= +\infty | X)>0$ a.s., and $P(T\in (\tau,+\infty) )  = 0 $.  
\item \label{cond:identification2} For any $\gamma \in \mathbb R^d$, $ \var( g(\gamma_0,X) / g(\gamma,X)) = 0$ implies that $\gamma=  \gamma_0 $. 
\setcounter{saveenum}{\value{enumi}}
\end{enumerate}

The following hypotheses (H\ref{cond:consistency_gamma_0}) and (H\ref{cond:consistency_gamma_1}) help to control the complexity of the underlying class of functions as well as to guarantee the continuity of the function to maximize. Let $|\cdot|_k$ denote the $\ell_k$-norm.

\begin{enumerate}[(H1)]\setcounter{enumi}{\value{saveenum}}
\item \label{cond:consistency_gamma_0} The true value $\gamma_0$ belongs to the interior of a compact set $B\subset \mathbb R^q$.
\item \label{cond:consistency_gamma_1} There exist functions $m_1:\mathcal S\rightarrow \mathbb R_{\geq 0}$ and $M_1:\mathcal S\rightarrow \mathbb R_{\geq 0}$ such that for every $x\in \mathcal S$ and every $\gamma\in B$, we have $0<m_1(x)\leq  g(\gamma, x)\leq M_1(x)$ and $E[|\log(m_1(X))|]$, $E[|\log(M_1(X))|]$ and $E[M_1^2(X)]$ are finite. There exists a function $c_1: \mathcal S \rightarrow \mathbb R_{\geq 0}$ such that for every $x\in\mathcal S$ and every $(\gamma,\tilde \gamma)\in B^2$,
\begin{align}\label{eq:mean_value_th1}
|  g(\gamma,x) -  g(\tilde \gamma,x) | \leq |\gamma - \tilde \gamma |_1 c_1(x),
\end{align}
with $0< E [c_1^2(X)]<+\infty $. 
\setcounter{saveenum}{\value{enumi}}
\end{enumerate}

\begin{proposition}\label{propositionconsistency}
Under (H\ref{cond:identification1})--(H\ref{cond:consistency_gamma_1}), we have that $\w\gamma \overset{\mathbb P}{\lr} \gamma_0$.
\end{proposition}

We now discuss assumption (H\ref{cond:identification2}) by considering some examples.

\begin{example}[Cox with cure]
When $g (\gamma,x)=\exp(\gamma^Tx)$, (H\ref{cond:identification2}) is equivalent to the statement that $ \var(X)$ has full rank.
\end{example}
 
Without specifying $g (\gamma,x)=\exp(\gamma^Tx)$, identifiability might not hold. Indeed, consider the case where $g (\gamma,x)=|\gamma^Tx|$, then of course different pairs $(\theta,\gamma)$ could lead to the same function $x\mapsto \theta | \gamma^Tx|$. A possibility when facing such difficulties is to restrict $\gamma$ to the unit sphere in $\R^{d}$. Then identifiability might be recovered. We refer to this model as a directional model.
 
\begin{example}[directional model]\label{ex:directionalmodel}
Suppose that $g (\gamma,x)=\eta(\gamma^Tx)$ with $|\gamma |_2=1$, $\gamma_1>0$. One can typically think of functions of the form $g (\gamma,x)= |\gamma^Tx|^k$, for some $k\geq 1$. Such models allow for a geometric interpretation in the same vein as the single-index models \citep{hardle2011}. The information available from the covariates $X$ to predict $Y$ is contained in the linear transformation $P_{\gamma}X$, where $P_\gamma$ stands for the orthogonal projector on $\text{span}(\gamma)$. 
For more details about identifiability of single-index models, we refer to Theorem 1 in \cite{lin+k:2007} as well as Theorem 1 in \cite{portier+d:2013} where $X$ is required to possess a density.
\end{example}

If $|\gamma|_2=1$ does not hold, then identifiability could fail unless more specific forms are considered for $\eta$. An example where identifiability is still satisfied is given below.

%

\begin{example}[Modified Cox]
An interesting choice is when $g (\gamma,x)=\exp(\rho_k(\gamma^T x))$, where $\rho_k(t) = \sign (t) |t|^k$,  for $k > 1$.  In the following lines, we obtain (H\ref{cond:identification2}) under the assumption that $X$ has a continuous density and  $B(0,r)$ is included in the support of $X$. Suppose that
$\rho_k (\gamma_0^Tx) - \rho_k (\gamma^Tx)$ is constant for almost every $x\in B(0,r)$. Suppose that $\gamma $ and $\gamma_0$ are linearly independent. Then, take $\alpha \in B(0,r)$ such that $\alpha^T\gamma = 0 \neq \alpha^T\gamma_0$. Let $g(x) = \rho_k(\gamma_0^Tx)- \rho_k(\gamma^Tx)$ and $K$ be a probability density function. For any $s \in (0,1)$ we have, using approximation theory, that $(g \star  K_h ) (s \alpha ) \to s^k \rho_k (\gamma_0^T\alpha) $   as  $h\to 0 $, where $K_h(\cdot) = K(\cdot/h)/h^d$. Hence for any $s\in  (0,1)$, $s^k \rho_k ( \gamma_0^T \alpha)$ is constant which is impossible. Supposing that $\gamma $ and $\gamma_0$ are linearly dependent, we directly obtain that  $\gamma =\gamma_0$.
\end{example}

\subsection{Asymptotic normality of $\w \gamma$}\label{sec:3_2}

We now introduce some notations that will be useful to express the asymptotic normality results. For every $y\in \mathbb R_{\geq 0}$, $\gamma\in  \mathbb R ^d$, $Q_{\gamma}(y) = E[g(\gamma,X)R(y)]$,
$d_\gamma (x)=  { \nabla_{\gamma} g(\gamma,x)}/{ g(\gamma,x)}$, $ h_\gamma(y) = 
{\nabla_\gamma  Q_\gamma(y)}/{ Q_\gamma(y)} $.
We define
\begin{align}
\label{eq:variance_I0} &I_0= \int  E \left[ \{d_0 (X) -  h_0(u)\} \{ d_0 (X)-h_0(u) \}^T g(\gamma_0,X)  R(u)\right] d\Lambda_0(u),
\end{align}
where $d_0 = d_{\gamma_0}$ and $h_0 = h_{\gamma_0}$.
We require the following assumptions to obtain an asymptotic decomposition for $\w \gamma$.
\begin{enumerate}[(H1)]\setcounter{enumi}{\value{saveenum}}
\item \label{cond:asymptotics_gamma_1} The matrix $I_0$ has full rank.
\item \label{cond:asymptotics_gamma_2} For every $x\in \mathcal S$, $\gamma\mapsto g(\gamma, x)$ is differentiable and there exists a function $c_2: \mathcal S \rightarrow \mathbb R_{\geq 0}$ such that for every $x\in\mathcal S$ and every $(\gamma,\tilde \gamma)\in B^2$,
\begin{align}\label{eq:mean_value_th2}
|  \nabla_\gamma g(\gamma,x) -  \nabla_\gamma  g(\tilde \gamma,x) |_1 \leq |\gamma - \tilde \gamma |_1 c_2(x),
\end{align}
with $ 0< E [c_2^2(X)]<+\infty $. Moreover there exists a function $M_2:\mathcal S\rightarrow \mathbb R_{\geq 0}$ such that, for every $x\in\mathcal S$, $|\nabla_\gamma g(\gamma,x)|_1 <  M_2(x)$ where  $E[M_2(X)]$, $E[ M_2^2(X) / m_1(X)  ]$, $ E[ (c_2(X)+M_2(X) )^2 M_1(X)/m_1^2(X)   ]$, and
 $ E[ M_2^2(X) (c_1(X)+ M_1(X))^2M_1(X)/m_1^4(X)   ]$ are finite.
 \setcounter{saveenum}{\value{enumi}}
\end{enumerate}

\begin{proposition}\label{proposition:weak_cv_gamma}
Under (H\ref{cond:identification1})--(H\ref{cond:asymptotics_gamma_2}), we have that
\begin{align}\label{result1}
n^{1/2} (\w\gamma -\gamma_0) =n^{-1/2} I_0^{-1}\sum_{i=1} ^n \int (d_0(X_i) -  h_0(u))dM_i(u) +o_{\mathbb P}(1),
\end{align}
and in particular, using Lemma \ref{lemma:weakcv}, see the Appendix, combined with (\ref{formula:quadratic_variation}), it holds that $n^{1/2} (\w\gamma -\gamma_0) \overset{\text{d}}{\lr}  \mathcal N (0, I_0^{-1})$.
\end{proposition}

\subsection{Weak convergence of $\w\Lambda$}\label{sec:3_3}

Based on the decomposition obtained for $\w \gamma$, we can now obtain a uniform representation of the process $\{n^{1/2} (\w \Lambda(y) - \Lambda(y)) : y\in\mathbb R_{\geq 0}\}$. This is the statement of the next Proposition. 

\begin{proposition}\label{proposition:strong_decomp_Lambda}
Under (H\ref{cond:identification1})--(H\ref{cond:asymptotics_gamma_2}), we have that
\begin{align}\label{result2}
\sup_{y\in \mathbb R_{\geq 0}} \left| n^{1/2}(\w \Lambda(y)-\Lambda_0(y)) - \left\{  n^{-1/2} \sum_{i=1}^n \int_0^y \frac{dM_i(u)}{  Q_{0}(u)} -   \int_0^y h_0(u)^T d\Lambda_0(u) ( n^{1/2}  (\w \gamma-\gamma_0)) \right\}  \right| = o_{\mathbb P}(1).									
\end{align}
In particular, using Lemma \ref{lemma:weakcv}, the two terms involved in the decomposition are asymptotically independent and $ n^{1/2} (\w \Lambda-\Lambda_0)$ converges weakly to a tight centered Gaussian process in $\ell^\infty(\R_{\geq 0})$ with covariance process given by
\begin{align*}
(y,y')\mapsto \int_0^{\min(y,y')}\frac{d\Lambda_0 (u)}{  Q_{0}(u)} + \left(\int_0^y h_0(u)^Td\Lambda_0(u)\right)  I_0^{-1}\left( \int_0^{y'} h_0(u)d\Lambda_0(u)\right).
\end{align*}
\end{proposition}

The two previous propositions, Proposition \ref{proposition:weak_cv_gamma} and \ref{proposition:strong_decomp_Lambda}, form the basis of the next analysis, which ultimately describes the estimator $\w p(x)$ of the cure proportion $p(x)$. The following results will be obtained as (almost direct) consequences of Propositions \ref{proposition:weak_cv_gamma} and \ref{proposition:strong_decomp_Lambda} and so are referred to as corollaries.

\subsection{Asymptotic normality of $\w \theta$} 
Since $\w\theta=\lim_{y\r +\infty} \w\Lambda(y) = \w\Lambda(\tau)$ and $\theta_0 = \Lambda_0(\tau)$, the weak convergence of $  n^{1/2} (\w \theta-\theta_0)$ is deduced from the weak convergence of $ n^{1/2} (\w \Lambda-\Lambda_0)$ as the finite dimensional laws converge in distribution. The expression for the asymptotic variance is deduced from the one given in Proposition \ref{proposition:strong_decomp_Lambda}.

\begin{corollary}
Under (H\ref{cond:identification1})--(H\ref{cond:asymptotics_gamma_2}),  $n^{1/2} (\w \theta-\theta_0)$ converges in distribution to a centered Gaussian distribution with variance
\begin{align}\label{express:asym_variance_theta}
v_\theta = \int \frac{d\Lambda_0(u)}{Q_0(u)} +\left( \int h_0(u)^T d\Lambda_0(u) \right)I_0^{-1}\left(\int h_0(u) d\Lambda_0(u)\right).
\end{align}
\end{corollary}

As $\w F = \w \Lambda / \w \theta$, invoking some Delta-method arguments, the weak convergence of the process $ n^{1/2} (\w F-F_0)$ can be established. This however is not needed in the following.

\subsection{Cure rate estimation}

Recall that the cure proportion associated to $x\in \mathcal S$ is given by $p_0(x)= \exp(-g(\gamma_0,x) \theta_0)$ and that the estimator is $\w p(x)= \exp(-g(\w \gamma,x) \w \theta)$.
A Taylor development gives that
\begin{align*}
n^{1/2} (\w p(x) - p_0(x) )  = -p_0(x) g(\gamma_0,x)\left \{n^{1/2} (\w \theta - \theta_0) + \theta_0d_0(x)^T n^{1/2} (\w\gamma-\gamma_0)\right\} +o_{\mathbb P}(1).
\end{align*}
Injecting (\ref{result1}) and (\ref{result2}) in the previous display leads to the following statement.
\begin{corollary}\label{prop:cure_proportion}
Under (H\ref{cond:identification1})--(H\ref{cond:asymptotics_gamma_2}), for a given $x\in \mathcal S$,  we have that
\begin{align*}
&n^{1/2} (\w p(x) - p_0(x) )  \\
&=  -p_0(x) g(\gamma_0,x) n^{-1/2} \sum_{i=1}^n \left\{ \int \frac{dM_i(u)}{  Q_{0}(u)} + u_0(x)^T  I_0^{-1} \int (d_0(X_i) -  h_0(u))dM_i(u)\right\}+o_{\mathbb P}(1),
\end{align*}
where $u_0(x)=\theta_0d_0(x) - \int h_0(u) d\Lambda_0(u) $. Consequently, $n^{1/2} (\w p(x) - p_0(x) )$ converges in distribution to a centered Gaussian distribution with variance 
\begin{align*}
& v_p(x) = p_0(x)^2 g(\gamma_0,x)^2 \left( \int \frac{d\Lambda_0(u)}{  Q_{0}(u)} + u_0(x)^T I_0^{-1} u_0(x) \right).
\end{align*}
\end{corollary}

Note that a similar result can be obtained concerning the estimator $\exp(-g(\w \gamma,x) \w \Lambda(y))$ of the survival function $S_0(y|x)$ but we prefer to omit this for the sake of brevity.

\section{Simulation study}\label{s6}

We performed some extensive Monte Carlo simulations in order to assess the performnce of our suggested estimators. The simulations were performed  under a variety of conditions on the censoring rate, sample size and cure rate. The data were generated according to the following model:
\begin{align}\label{def:model_simu}
S(t|x_1,x_2) =\exp\big[-\exp\big\{\Gamma(\gamma_{01}x_1+\gamma_{02} x_2)\big\} \theta_0 F_0(t)\big].
\end{align}
In the above model, we chose the link function $\Gamma(\cdot)$ to be either the identity, the cubic or the sine function. For clarity, in the first part of this simulation study we will focus on the case of the identity function. With few exceptions, all our comments and findings also apply to the case where $\Gamma(\cdot)=(\cdot)^3$ and $\Gamma(\cdot)=\sin(\cdot)$. In all our simulations, $\log(\theta_0)=0.1$, $\gamma_{01}=-2$, $\gamma_{02}=1$, $F_0$ is the cumulative distribution function of a uniform variable on $[0, 1]$, $X_1$ is a uniformly distributed random variable on $[\alpha, \alpha + 1]$, $X_2$ is a normal random variable with mean $\alpha$ and standard deviation $1/12$, and $X_1$ and $X_2$ are independent. The censoring variable is exponential with parameter $\lambda$ and is independent of $(X_1,X_2)$. By varying the latter we mainly control the censoring rate, while by varying $\alpha$ we control the cure rate.

Suppose we have a sample $(Y_i,\delta_i,X_i)$, $i=1,\ldots,n$ from the distribution described above, with $X_i=(X_{i1},X_{i2})^T$.
We obtain $\w{\gamma}=(\w{\gamma}_1,\w{\gamma}_2)^T$, the estimator of $\gamma_0=(\gamma_{01},\gamma_{02})^T$, by maximizing the partial likelihood function given by (\ref{NPMLEcoxmodel_gamma_Lambda}), using the Newton-Raphson algorithm. We get $\w{\theta}$, the estimator of $\theta_0$, by applying (\ref{NPMLE_coxmodel_theta_F}). The cure probability estimator is then obtained by
$$
\w{p}(x_1,x_2)=\exp\big[-\exp\big\{\Gamma(\w{\gamma}_1x_1+\w{\gamma}_2x_2)\big\}\w{\theta}\big].
$$      
Using the plug-in principle together with (\ref{eq:variance_I0}), we obtain an estimator for the asymptotic variance-covariance matrix of $\w{\gamma}$ which is given by $\w{I}^{-1}/n$, where 
\begin{align*}
\w{I}= n^{-1} \sum_{i=1}^n \delta_i  \left\{ (d_{\widehat \gamma} (X_i )  -\w{h}_{\widehat \gamma} (Y_i) )( d_{\widehat \gamma} (X_i )- \w{h}_{\widehat \gamma} (Y_i)) ^T\right\},
\end{align*}
with, for every $\gamma\in B$, $\w h_\gamma(y)= {\nabla_{\gamma} \w Q_\gamma(y)}/{\w Q_\gamma(y)}$.
Similarly, using (\ref{express:asym_variance_theta}), we obtain an estimator of the asymptotic variance of $\w{\theta}$ given by $\w v_\theta/n$, where 
\begin{align*}
\w v_\theta = n^{-1}\sum_{i=1}^n \frac{\delta_i}{\w Q_{\w \gamma}(Y_i)^2} + \left( n^{-1}\sum_{i=1}^n \frac{\delta_i \w h_{\w \gamma} (Y_i)}{\w Q_{\w \gamma}(Y_i)}  \right)^T \w I^{-1} \left( n^{-1}\sum_{i=1}^n \frac{\delta_i \w h_{\w \gamma} (Y_i)}{\w Q_{\w \gamma}(Y_i)}  \right).
\end{align*} 
And using the expression for the variance of $\w p$ given in Corollary \ref{prop:cure_proportion}, we obtain an estimator of the asymptotic variance of $\w p(x)$ given by $\w v_p /n$, where 
\begin{align*}
\w v_p = \w p(x)^2 g(\w \gamma,x)^2 \left( n^{-1}\sum_{i=1}^n \frac{\delta_i}{\w Q_{\w \gamma}(Y_i)^2}  +\w u(x)^T \w I^{-1} \w u(x) \right),
\end{align*}
with $\w u(x) = \w \theta d_{\w \gamma}(x) - n^{-1}\sum_{i=1}^n {\delta_i\w h_{\w \gamma} (Y_i)}/{\w Q_{\w \gamma}(Y_i)}$ and $x=(x_1,x_2)^T$.

We perform $N = 2000$ replications for four sample sizes ($n = 100$, $n = 200$, $n=400$ and $n=600$), three levels
of censoring ($20\%$, $40\%$ and $60\%$) and three levels of cure ($20\%$, $40\%$ and $60\%$). For every scenario and every replication, we calculate the estimators $\w{\gamma}_1$, $\w{\gamma}_2$, $\w{\theta}$ and $\w{p}(x_1,x_2)$ together with their estimated asymptotic variance ($\widehat{AVar}$) and the corresponding asymptotic $95\%$ confidence intervals based on the asymptotic normality. Based on the $2000$ replications, we also calculate the empirical bias, the empirical variance ($VAR$), the empirical mean squared error ($MSE$) of every estimator together with the empirical coverage probability ($COV$) for the confidence intervals. In the case of the cure probability $p(x_1,x_2)$ we did the calculations for $x_2=0$ and every quantile of $X_1$ corresponding to the probability levels $0.01,0.02,\ldots,0.99.$ We summarize the results by taking the average of the resulting $99$ empirical $VAR$'s, empirical $MSE$'s and empirical $COV$'s. Due to space limitations, we provide below only some selected but representative scenarios. 

\begin{figure}[H]
	\centering
	
	\begin{subfigure}{1\textwidth}
		\includegraphics[width=.95\textwidth]{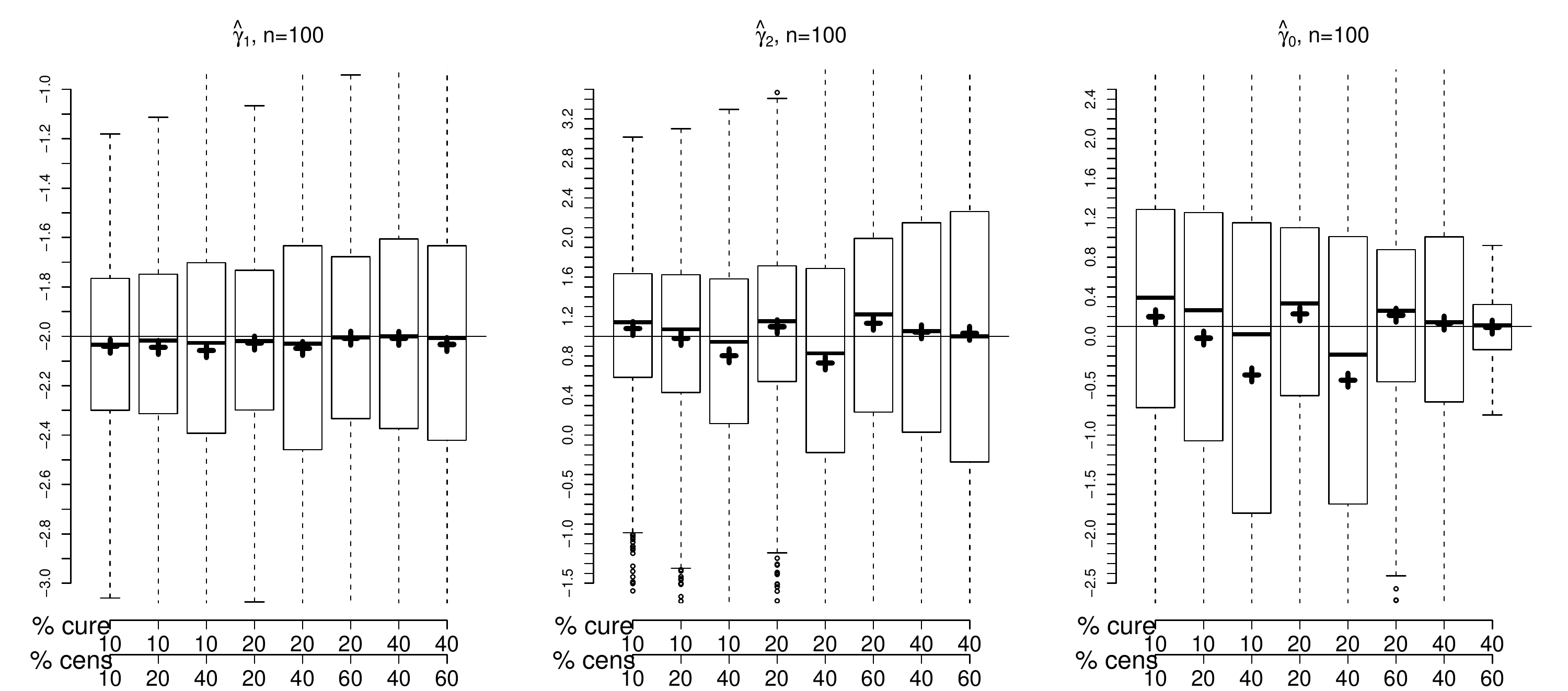}
	\end{subfigure}
	
	\par\smallskip 
	
	\begin{subfigure}{1\textwidth}
		\includegraphics[width=.95\textwidth]{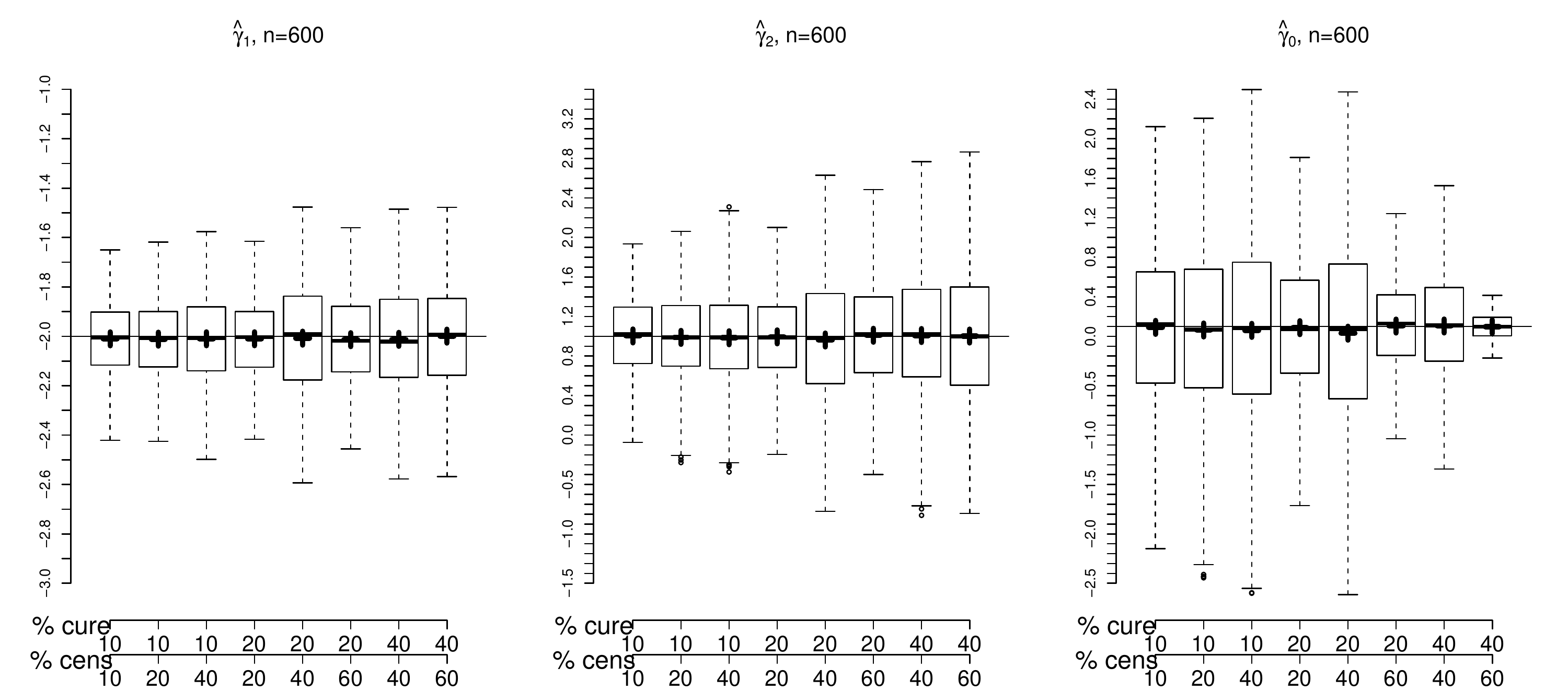}
	\end{subfigure}
	
	\caption{Boxplots of $\w{\gamma}_1$, $\w{\gamma}_2$ and $\w{\gamma}_0$ for $n=100$ and $n=600$ and for $\Gamma(\cdot)=\cdot$. The empirical mean of the estimates is indicated by a $+$. The true values are indicated by a horizontal line.}  
	\label{fig:fig1e}
\end{figure}

Figure \ref{fig:fig1e} provides the boxplots for $\w{\gamma}_1$, $\w{\gamma}_2$ and $\w{\gamma}_0=\log(\w{\theta})$. By comparing
the upper and lower part ($n=100$ vs $n=600$) of this figure, we clearly see that the performance of the estimators improves with increasing sample size both in terms of bias and variance. This confirms the consistency of these estimators. This figure also shows the effect of the cure rate and the censoring rate. As expected, increasing the latter rates results in a larger bias and, especially, in a larger variance of the estimators. This effect can also be seen in Figure \ref{fig:fig2} which provides the boxplots for the asymptotic estimated variances. Compared to the censoring rate, the cure rate seems to have no, or very limited, effect on $\w{\gamma}_1$ and $\w{\gamma}_2$, but it does affect the bias and the variance of $\w{\gamma}_0$. In fact, when the percentage of cure increases, the bias and the variance of $\w{\gamma}_0$ decrease (and so does the MSE).  Globally, it seems that the estimation of $\w{\gamma}_0$ is more difficult than the estimation of $\w{\gamma}_1$ and $\w{\gamma}_2$. This is especially the case when the censoring percentage is large and the cure probability is small. If moreover the sample size is small, then the bias can be quite large.  

\begin{figure}[H]
	\centering
	\begin{subfigure}{1\textwidth}
		\includegraphics[width=1\textwidth]{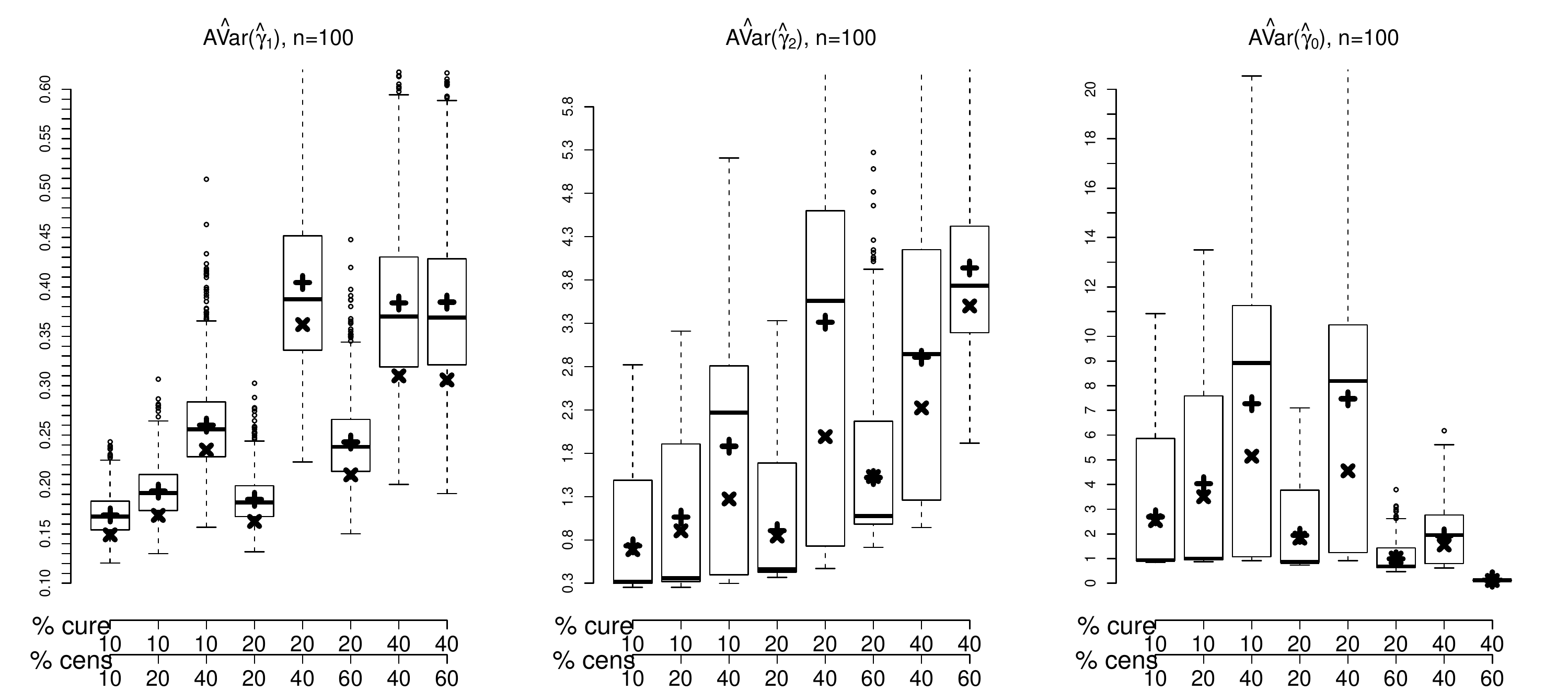}
	\end{subfigure}
	
	\par\smallskip 
	
	\begin{subfigure}{1\textwidth}
		\includegraphics[width=\textwidth]{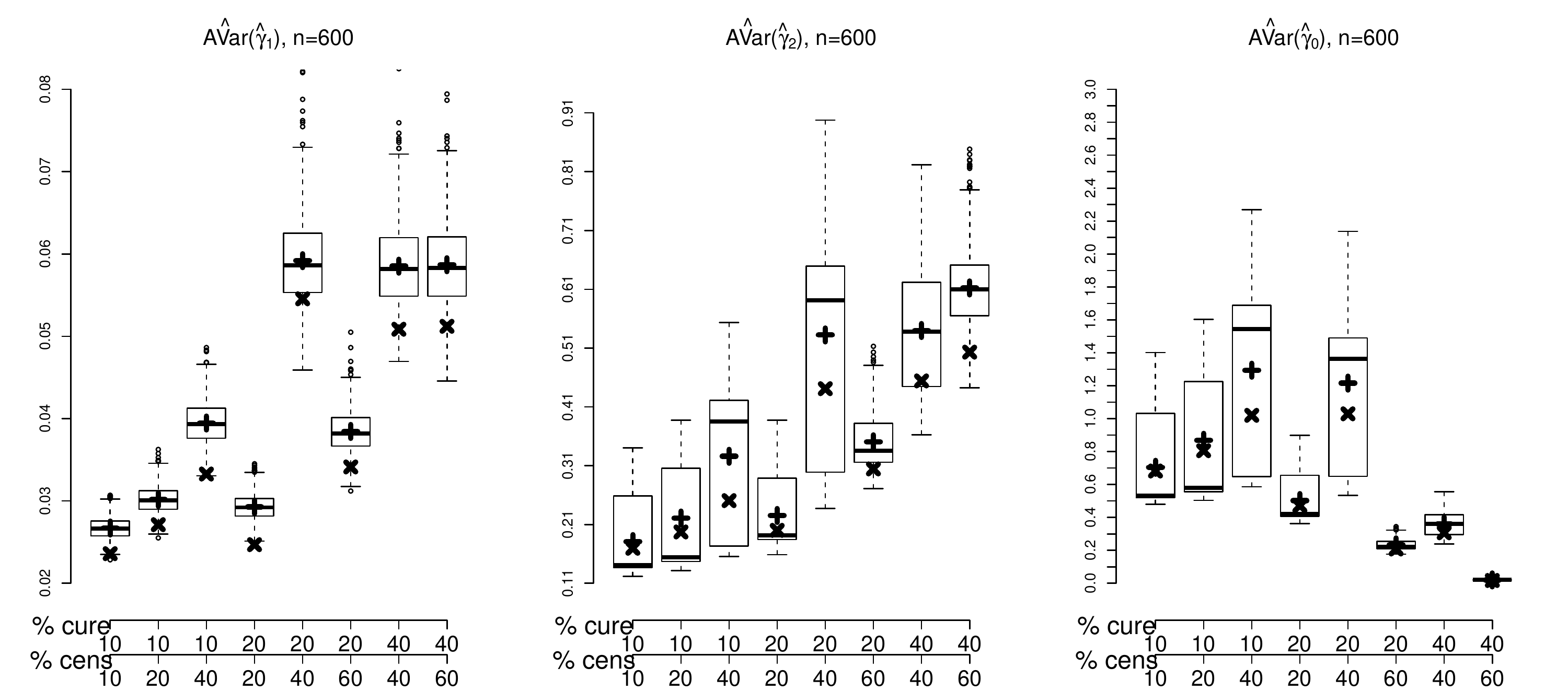}
	\end{subfigure}
	\caption{Boxplots of $\widehat{AVar}(\w{\gamma}_1)$, $\widehat{AVar}(\w{\gamma}_2)$ and $\widehat{AVar}(\w{\gamma}_0)$ for $n=100$ and $n=600$ and for $\Gamma(\cdot)=\cdot$. The empirical mean of $\widehat{AVar}$ is indicated by a $+$, the empirical variance of the estimates ($\w{\gamma}_1,\w{\gamma}_2,\w{\gamma}_0$) is indicated by a $\times$.}  
	\label{fig:fig2}
\end{figure}

As we said before, Figure \ref{fig:fig2} provides the boxplots for the asymptotic estimated variances. The plots suggest that the proposed estimators are consistent (note that the $y$-axis in the upper and the lower plots do not have the same scale). Basically the remarks we made above on the effect of the proportion of cure and censoring remain valid for the proposed estimators of the variances. Again, it can be seen that estimating the variance of $\w{\gamma}_0$ is more difficult and can lead to, relatively, large variances especially when the censoring and cure rates are large and the sample size is small.   

Figure \ref{fig:fig3e} which provides some Q-Q plots for the estimated parameters confirms the validity of the normal approximation of the sampling distributions of $\w{\gamma}_1$ and $\w{\gamma}_2$. However, this approximation seems to be less accurate for $\w{\theta}$ even when $n=600$ (figure not shown here). In fact the sampling distribution of the latter tends to be positively skewed especially when the censoring rate is large. Applying the logarithmic transformation, seems to solve the problem as it makes the distribution more symmetric (see the Q-Q plot for $\w{\gamma}_0$ in Figure \ref{fig:fig3e}).

\begin{figure}[H]
	\centering
	\begin{subfigure}{0.49\textwidth}
		\includegraphics[width=1\textwidth]{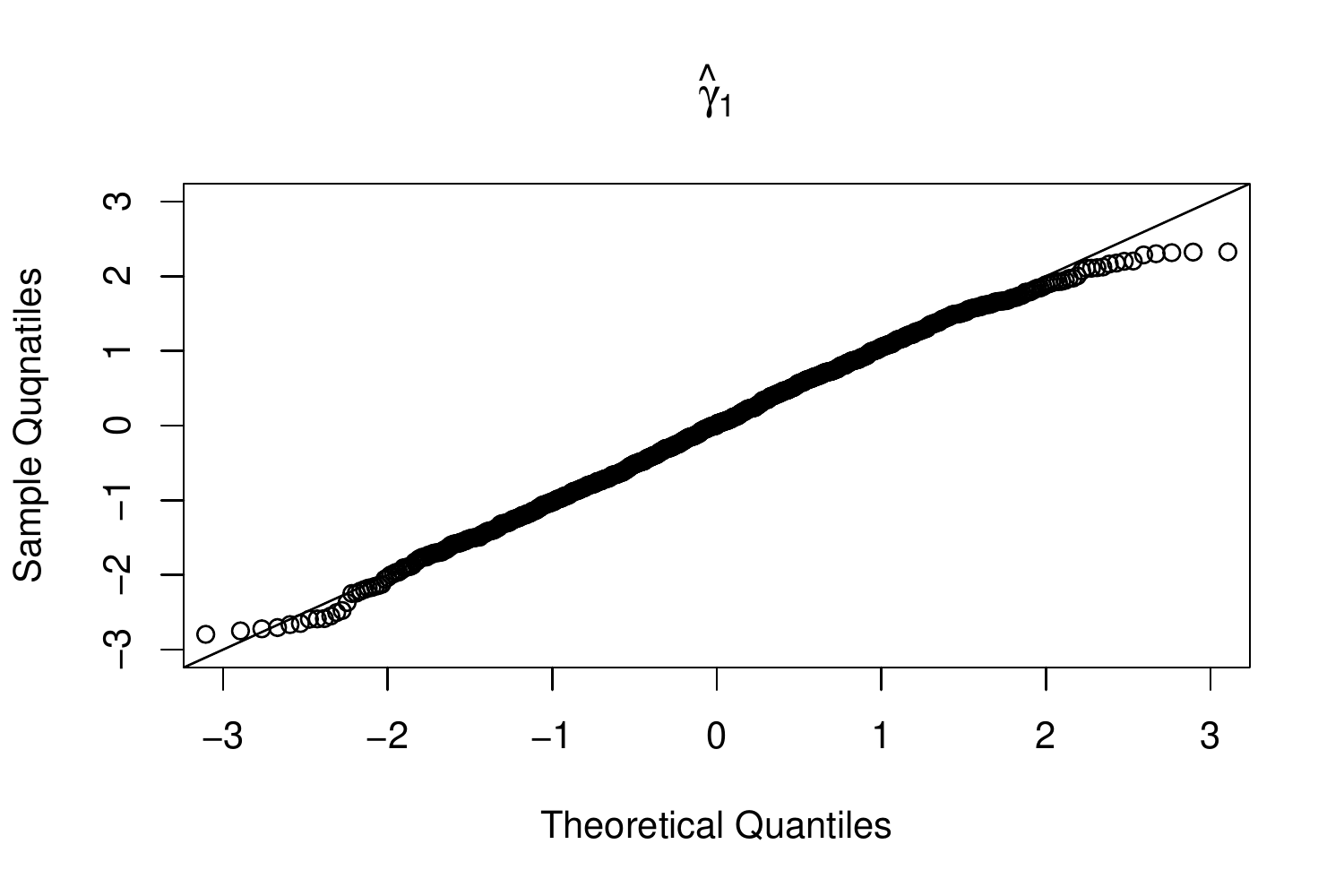}
	\end{subfigure}
	\begin{subfigure}{0.49\textwidth}
		\includegraphics[width=1\textwidth]{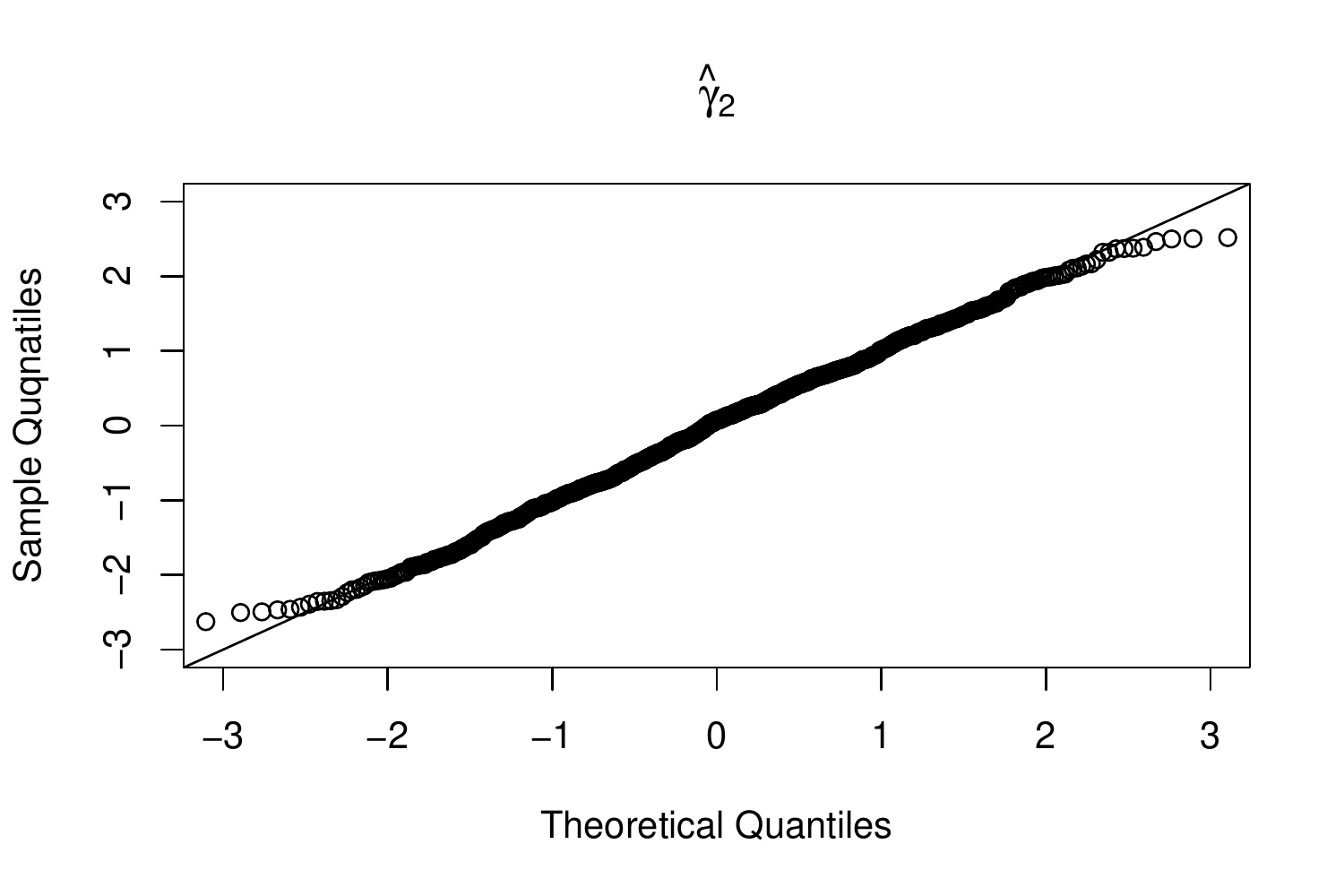}
	\end{subfigure}
	\begin{subfigure}{0.49\textwidth}
		\includegraphics[width=1\textwidth]{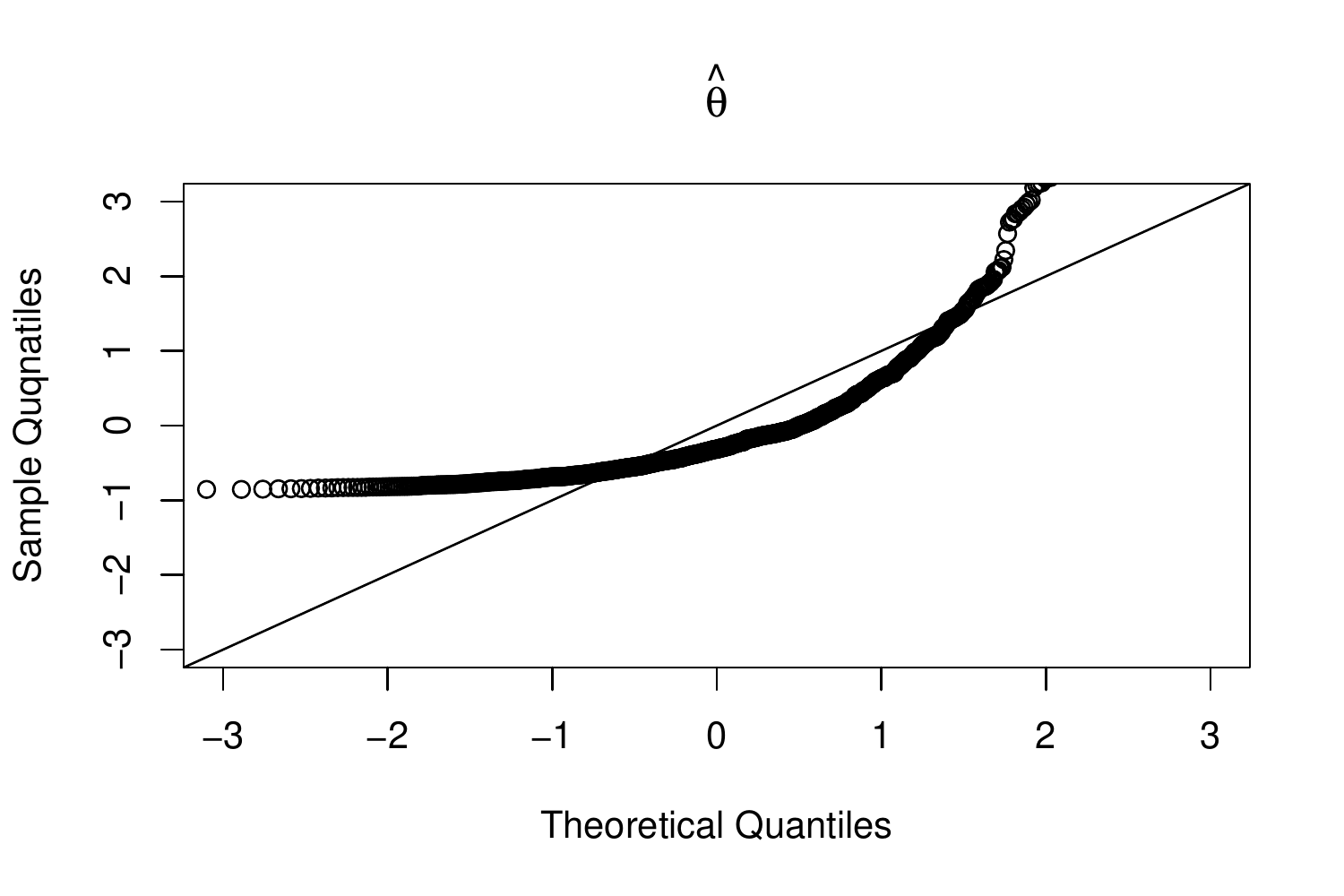}
	\end{subfigure}
	\begin{subfigure}{0.49\textwidth}
		\includegraphics[width=1\textwidth]{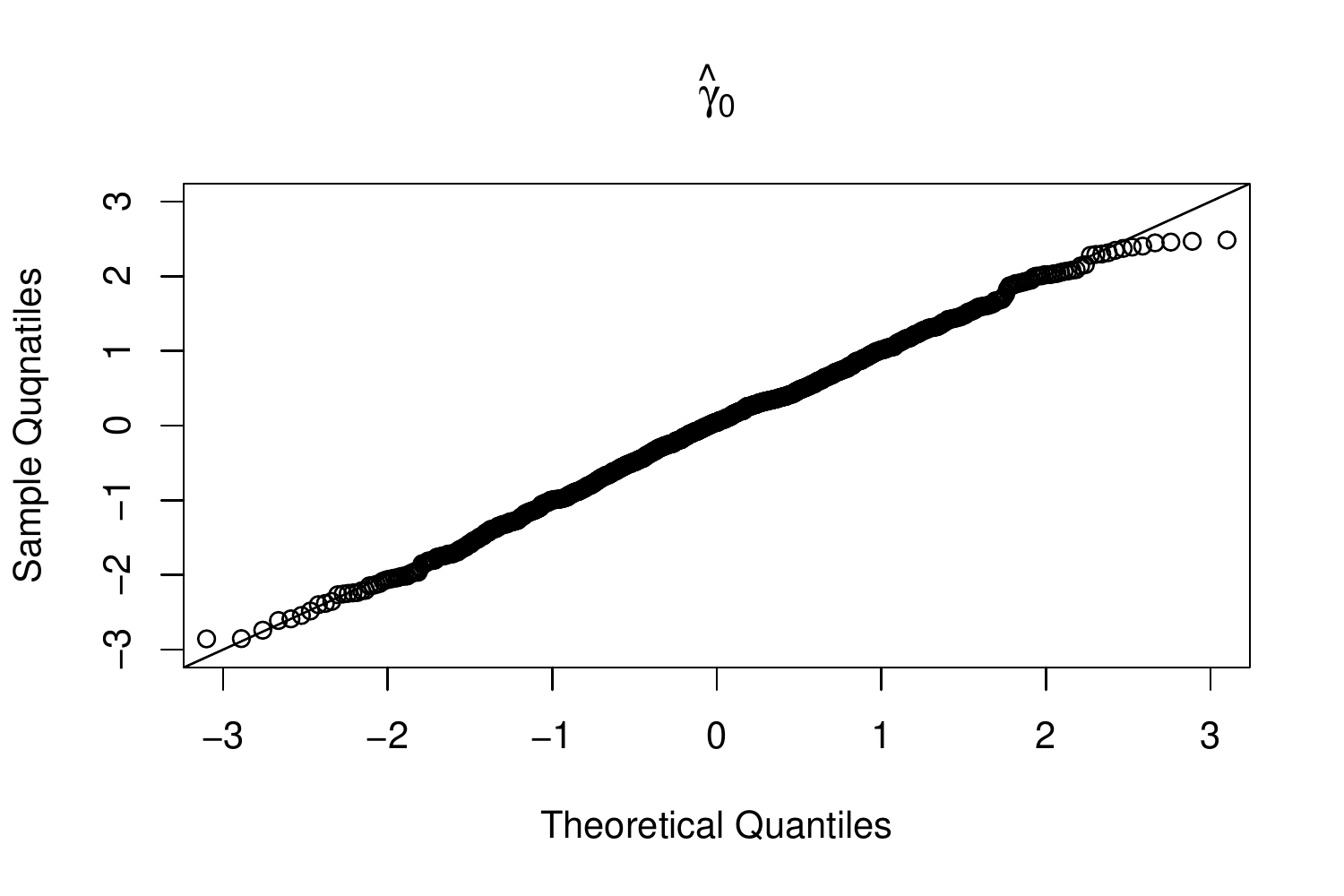}
	\end{subfigure}
	\caption{Normal Q-Q plot for the estimates ($\w{\gamma}_1, \w{\gamma}_2, \w{\theta}, \w{\gamma}_0$) for $n=100$ and for $\Gamma(\cdot)=\cdot$. The proportion of censoring and the cure rate both equal $0.40$.}  
	\label{fig:fig3e}
\end{figure}

In Table \ref{tab:tab1} we give the MSE and the variance for some of the studied scenarios and for $\Gamma(\cdot)=\cdot$, $\Gamma(\cdot)=(\cdot)^3$ and $\Gamma(\cdot)=\sin(\cdot).$ It is clear from these results that the variance is the dominant component in the mean squared errors. It can also be observed that the obtained results with the link $\Gamma(\cdot)=\cdot$ and $\Gamma(\cdot)=(\cdot)^3$ are globally better than the corresponding results obtained with $\Gamma(\cdot)=\sin(\cdot).$ Table \ref{tab:tab1} also shows the coverage probabilities ($COV$) of the $95\%$ asymptotic confidence intervals for the parameters $\gamma_1$, $\gamma_2$ and $\gamma_0=\log(\theta)$. The confidence intervals for the latter are based on the asymptotic normality of $\hat{\theta}$ and the Delta method. Globally, the obtained COV's are close to the nominal level. With $n=100$, the confidence intervals tend to be liberal when $\Gamma(\cdot)=\sin(\cdot)$ especially for $\gamma_0$. This also happens for $\gamma_1$ when $\Gamma(\cdot)=(\cdot)^3$.

\begin{table}[H]%
\centering
{\footnotesize
	\begin{tabular}{|c|c|c|ccc|ccc|ccc|}
		\toprule
		&       &       & \multicolumn{3}{c|}{MSE} & \multicolumn{3}{c|}{VAR} & \multicolumn{3}{c|}{COV} \\
		\toprule
		$n$   & $\%cure$ & $\%cens$ & $\gamma_{1}$ & $\gamma_{2}$ & $\gamma_0$ & $\gamma_{1}$ & $\gamma_{2}$ & $\gamma_0$ & $\gamma_{1}$ & $\gamma_{2}$ & $\gamma_0$ \\
		\toprule
		\multicolumn{12}{c}{$\Gamma(\cdot)=\cdot$} \\
		\toprule
		100   & 10    & 20    & 0.171 & 0.905 & 3.516 & 0.169 & 0.905 & 3.502 & 0.962 & 0.902 & 0.879 \\
		100   & 20    & 20    & 0.163 & 0.853 & 1.841 & 0.162 & 0.844 & 1.826 & 0.970 & 0.918 & 0.905 \\
		100   & 20    & 40    & 0.233 & 1.304 & 2.898 & 0.231 & 1.301 & 2.881 & 0.962 & 0.913 & 0.896 \\
		100   & 40    & 40    & 0.210 & 1.545 & 1.034 & 0.210 & 1.527 & 1.021 & 0.973 & 0.932 & 0.923 \\
		100   & 40    & 60    & 0.310 & 2.320 & 1.550 & 0.310 & 2.318 & 1.549 & 0.980 & 0.927 & 0.924 \\
		\hline
		600   & 10    & 20    & 0.027 & 0.197 & 0.806 & 0.027 & 0.197 & 0.805 & 0.968 & 0.948 & 0.941 \\
		600   & 20    & 20    & 0.025 & 0.200 & 0.471 & 0.025 & 0.200 & 0.471 & 0.968 & 0.956 & 0.952 \\
		600   & 20    & 40    & 0.034 & 0.292 & 0.666 & 0.034 & 0.292 & 0.665 & 0.972 & 0.954 & 0.942 \\
		600   & 40    & 40    & 0.034 & 0.304 & 0.208 & 0.034 & 0.304 & 0.208 & 0.968 & 0.962 & 0.956 \\
		600   & 40    & 60    & 0.051 & 0.454 & 0.306 & 0.051 & 0.454 & 0.306 & 0.969 & 0.964 & 0.956 \\
		\toprule
		\multicolumn{12}{c}{$\Gamma(\cdot)=(\cdot)^3$} \\
		\toprule
		100   & 10    & 20    & 0.050 & 0.044 & 0.112 & 0.050 & 0.044 & 0.107 & 0.955 & 0.951 & 0.967 \\
		100   & 20    & 20    & 0.110 & 0.102 & 0.047 & 0.108 & 0.101 & 0.047 & 0.875 & 0.873 & 0.946 \\
		100   & 20    & 40    & 0.128 & 0.123 & 0.084 & 0.127 & 0.122 & 0.083 & 0.882 & 0.879 & 0.966 \\
		100   & 40    & 40    & 0.227 & 0.067 & 0.059 & 0.205 & 0.067 & 0.059 & 0.899 & 0.921 & 0.953 \\
		100   & 40    & 60    & 0.299 & 0.131 & 0.111 & 0.266 & 0.130 & 0.111 & 0.895 & 0.913 & 0.955 \\
		\hline
		600   & 10    & 20    & 0.006 & 0.005 & 0.016 & 0.006 & 0.005 & 0.016 & 0.958 & 0.962 & 0.954 \\
		600   & 20    & 20    & 0.019 & 0.018 & 0.007 & 0.018 & 0.017 & 0.007 & 0.929 & 0.923 & 0.951 \\
		600   & 20    & 40    & 0.023 & 0.022 & 0.012 & 0.022 & 0.021 & 0.012 & 0.936 & 0.937 & 0.956 \\
		600   & 40    & 40    & 0.036 & 0.006 & 0.008 & 0.035 & 0.006 & 0.008 & 0.944 & 0.933 & 0.941 \\
		600   & 40    & 60    & 0.062 & 0.010 & 0.014 & 0.059 & 0.010 & 0.014 & 0.935 & 0.924 & 0.933 \\
		\toprule
		\multicolumn{12}{c}{$\Gamma(\cdot)=\sin(\cdot)$} \\
		\toprule
		100   & 10    & 20    & 0.656 & 0.532 & 0.270 & 0.656 & 0.532 & 0.204 & 0.915 & 0.908 & 0.914 \\
		100   & 20    & 20    & 0.625 & 0.353 & 0.173 & 0.538 & 0.248 & 0.132 & 0.949 & 0.959 & 0.897 \\
		100   & 20    & 40    & 0.942 & 0.582 & 0.222 & 0.790 & 0.411 & 0.173 & 0.945 & 0.946 & 0.867 \\
		100   & 40    & 40    & 0.988 & 0.547 & 0.138 & 0.737 & 0.522 & 0.136 & 0.954 & 0.837 & 0.833 \\
		100   & 40    & 60    & 1.614 & 0.708 & 0.161 & 1.197 & 0.666 & 0.156 & 0.932 & 0.849 & 0.871 \\
		\hline
		600   & 10    & 20    & 0.104 & 0.085 & 0.007 & 0.104 & 0.085 & 0.007 & 0.979 & 0.977 & 0.965 \\
		600   & 20    & 20    & 0.088 & 0.029 & 0.027 & 0.087 & 0.026 & 0.024 & 0.946 & 0.979 & 0.970 \\
		600   & 20    & 40    & 0.124 & 0.044 & 0.042 & 0.121 & 0.038 & 0.037 & 0.937 & 0.971 & 0.967 \\
		600   & 40    & 40    & 0.088 & 0.126 & 0.044 & 0.073 & 0.126 & 0.044 & 0.987 & 0.910 & 0.875 \\
		600   & 40    & 60    & 0.144 & 0.182 & 0.064 & 0.116 & 0.182 & 0.062 & 0.982 & 0.897 & 0.857 \\
		\bottomrule
	\end{tabular}%
	\caption{Empirical mean squared error (MSE), empirical variance (VAR) and empirical coverage probability (COV) for nominal $95\%$ confidence intervals for $\gamma_{1}$, $\gamma_{2}$ and $\gamma_0$.}
	\label{tab:tab1}
	}
\end{table}%

Figure \ref{fig:fig4} shows the empirical coverage probabilities (COV) of the confidence intervals for $p(x_1,x_2)$. We can see that these COV's can be quite unsatisfactory especially in the left tail of the support of $X_1$ even when the sample size is relatively large. To correct for this, we apply the logit transformation and the Delta method to construct confidence intervals for $\log(p/(1-p))$ and transform back (taking the logistic transformation) to get confidence intervals for the cure probabilities. This leads to very satisfactory results with coverage probabilities close to the nominal level both in the middle and in the tails especially when the sample size is large.

\begin{figure}[H]
	\centering
	\begin{subfigure}{0.49\textwidth}
		\includegraphics[width=1\textwidth]{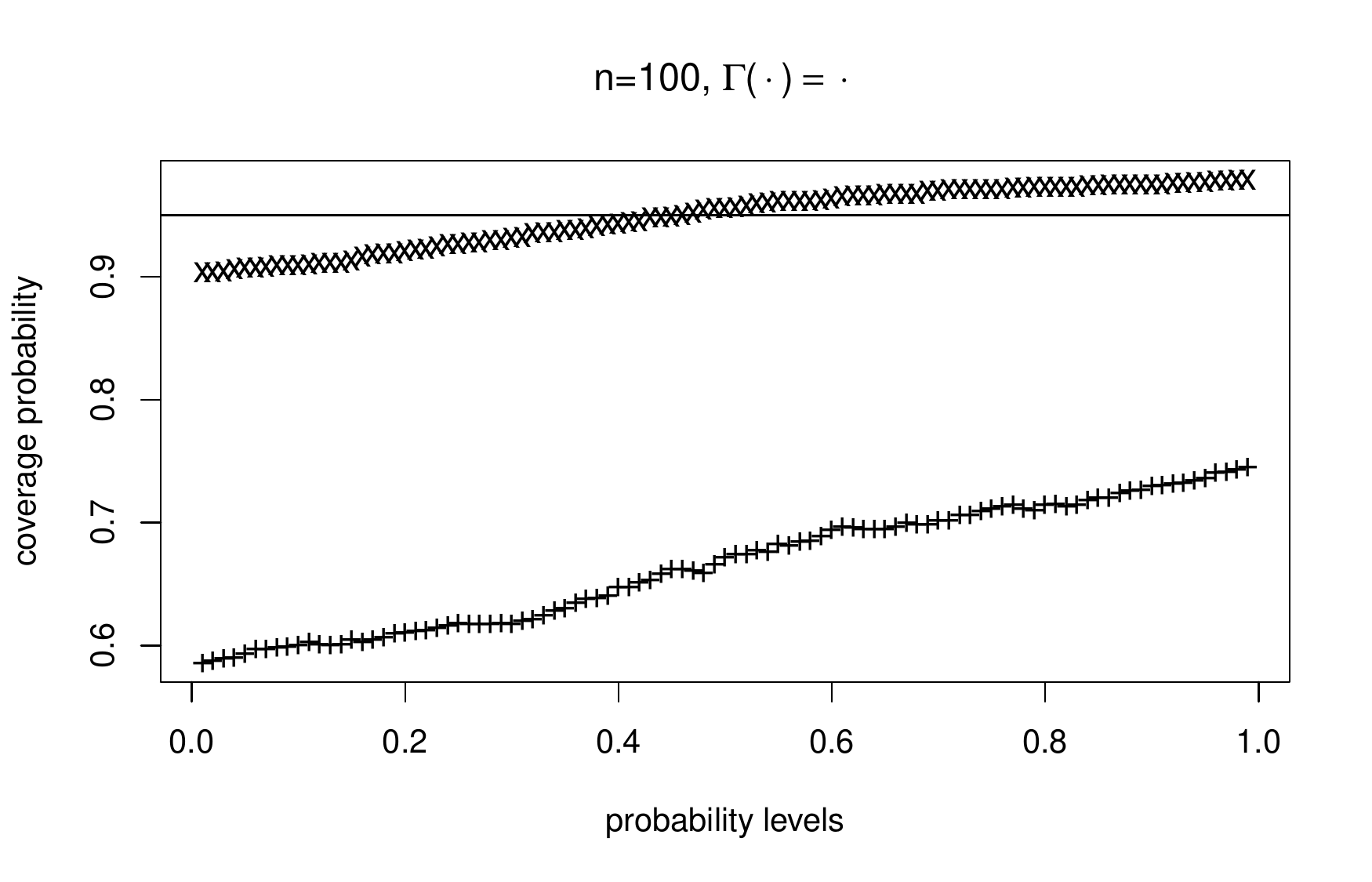}
	\end{subfigure}
	\begin{subfigure}{0.49\textwidth}
	\includegraphics[width=1\textwidth]{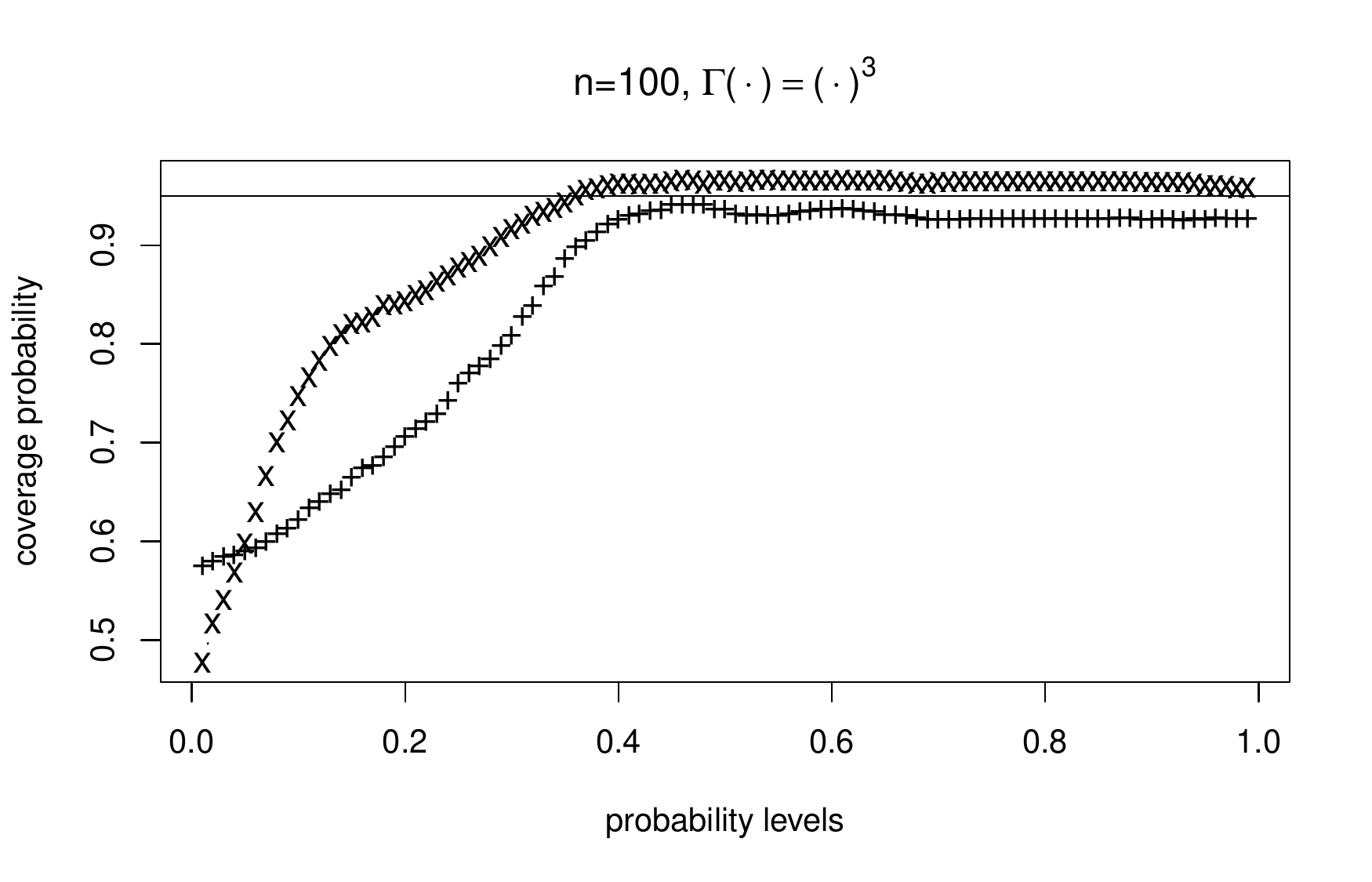}
\end{subfigure}
\begin{subfigure}{0.49\textwidth}
	\includegraphics[width=1\textwidth]{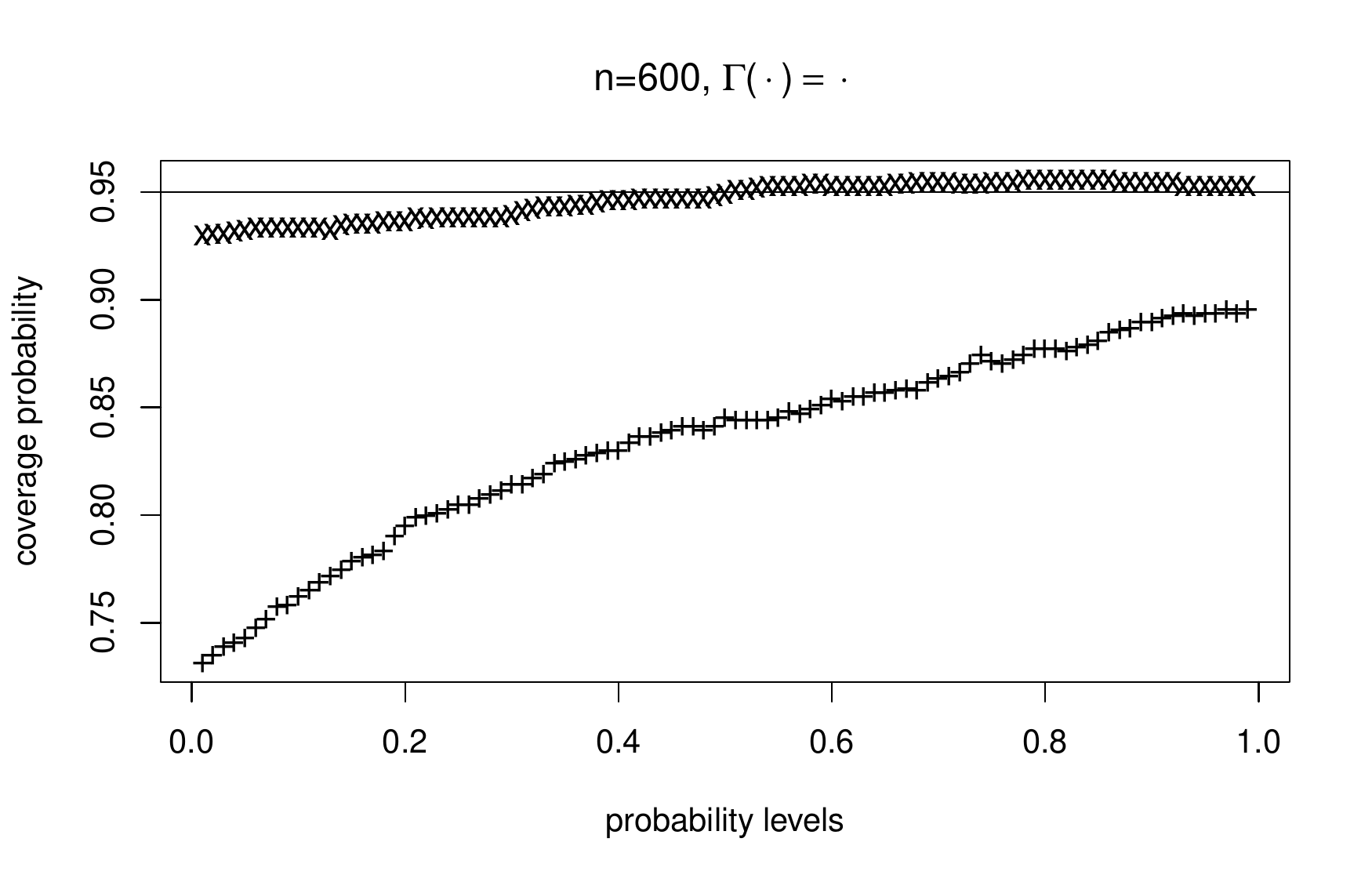}
\end{subfigure}
\begin{subfigure}{0.49\textwidth}
	\includegraphics[width=1\textwidth]{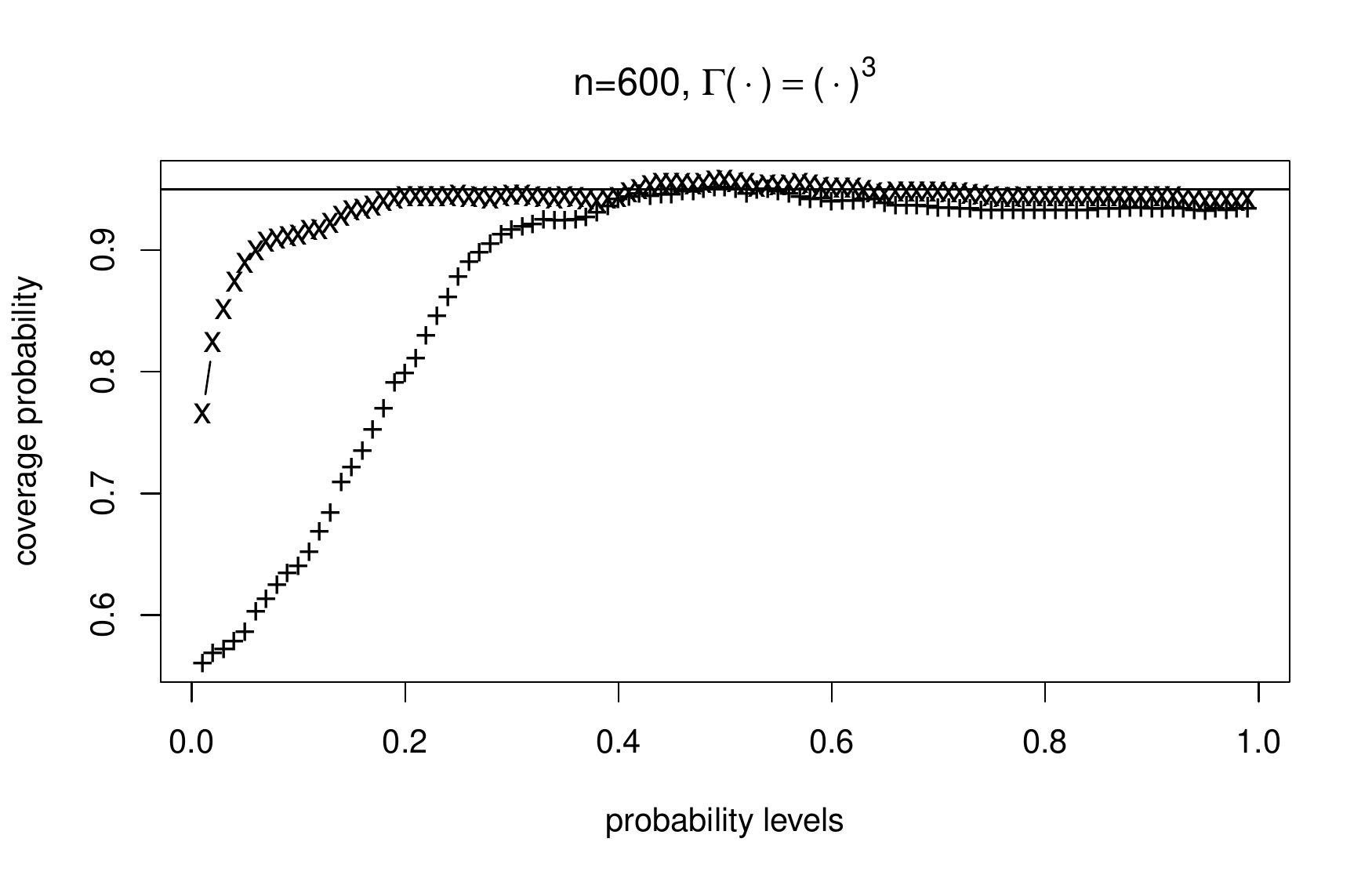}
\end{subfigure}
	\caption{Empirical coverage probabilities of nominal $95\%$ confidence intervals for the cure probability as a function of $x_1$  for $x_2=0$.  The coverage probabilities obtained without transformation are indicated by a $+$, those obtained after a logit transformation are indicated by a $\times$. The proportion of censoring and the cure rate both equal $0.40$.}  
	\label{fig:fig4}
\end{figure}

\section{Real data application} 

To illustrate the application of our model, the proposed methodology is applied on a real data set from a breast cancer study. The dataset consists of 286 patients that experienced a lymph-node-negative breast cancer between 1980 to 1995 \citep{Wang:2005}.  The event time of interest is the time to distant metastasis (DM). Among the 286 patients, 107 experienced a relapse from breast cancer. As can be seen from Figure \ref{fig:fig5}, the Kaplan-Meier estimator of the survival function shows a large plateau at about $0.60$. Furthermore,
$ 88 \% $ of the censored observations are in the plateau. A cure model seems therefore appropriate for these data. 

\begin{figure}[H]
\centering
\includegraphics[width=0.6\textwidth]{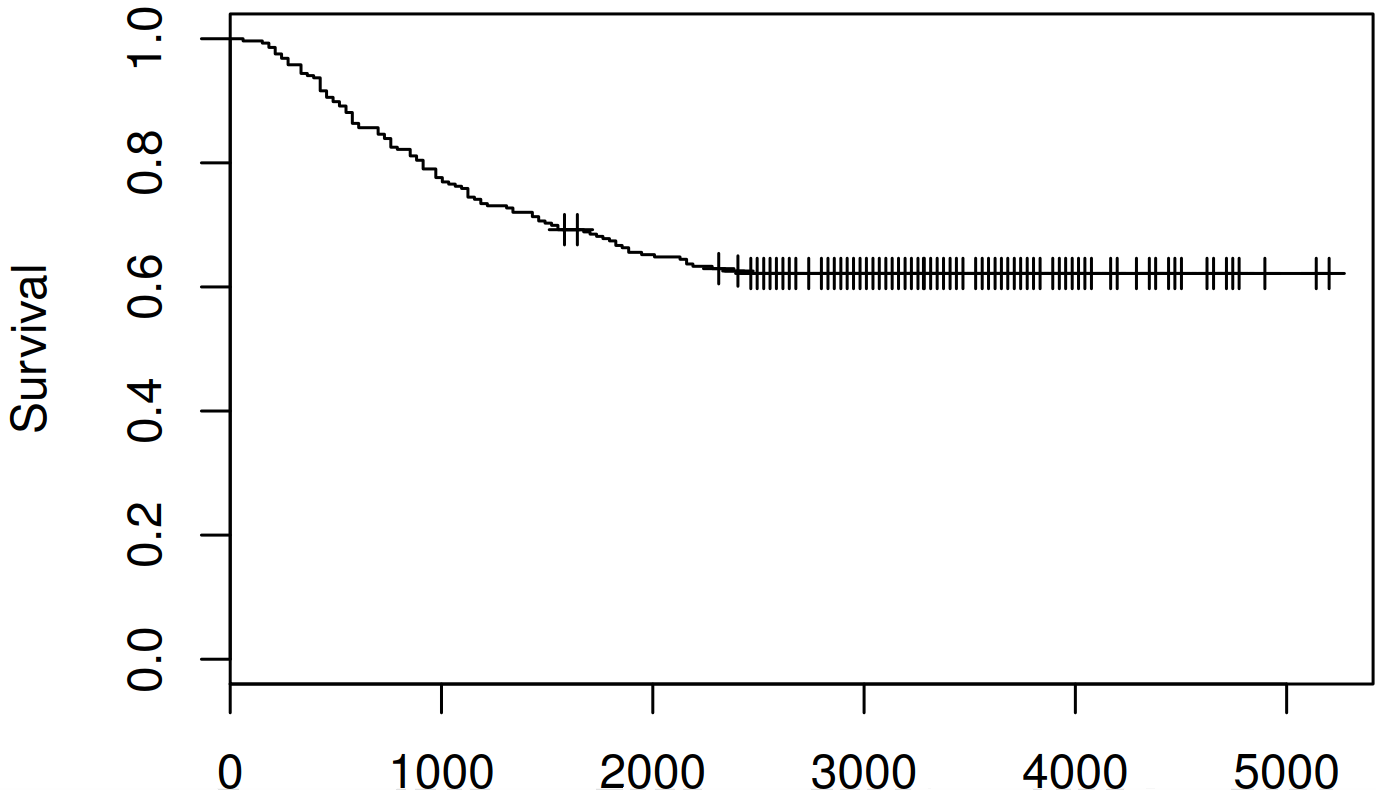}
	\caption{Kaplan-Meier estimator of the survival curve for time to distant metastasis for breast cancer survival data (censored observations are indicated by $+$).}  
	\label{fig:fig5}
\end{figure}

We consider two covariates : the age of the patient (ranges from 26 to 83 with a median of 52 years) and the estrogen receptor (ER) status, which is a binary variable equaling 0 ($ER-$) in the case of less than $10$ fmol per mg protein (77 patients in total) and equaling 1 ($ER+$) when $10$ fmol per mg protein or more (209 patients in total). We analyse the data using the semiparametric model given in (\ref{def:model_simu}) and  we choose the link function $\Gamma(x)$ to be either $x^k$ or $\sin(x^k)$ with $k=1,\ldots,8$. In Table \ref{tab:tab2} we report the values of the obtained profile log-likelihood (PLL) as given by (\ref{NPMLEcoxmodel_gamma_Lambda}), and the obtained full log-likelihood (FLL) as given by (\ref{maximumlikelihood}). Based on this result we can conclude that, in terms of the likelihood, the model that fits best these data is the model with the sine function and $k=4$.

\begin{table}[H]
	\centering
	\begin{tabular}{c|c|c|c|c|c|c|c|c|c|c|c}
		\multicolumn{6}{c|}{$\Gamma(x)=x^k$} & \multicolumn{6}{c}{$\Gamma(x)=\sin(x^k)$} \\ \hline
		\multicolumn{3}{c|}{$k$ odd} & \multicolumn{3}{c|}{$k$ even} & \multicolumn{3}{c|}{$k$ odd} & \multicolumn{3}{c}{$k$ even} \\ \hline
		$k$ & PLL             & FLL             & $k$  & PLL               & FLL     & $k$ & PLL             & FLL             & $k$  & PLL               & FLL             \\ \hline
		1   & 25.0            & -687.1           & 2   & 25.9            & -686.2  & 1 & 25.3              & -686.8           & 2 & 26.6              & -685.5              \\ 
		3   & 25.2            & -686.9           & 4   & 25.9            & -686.2  & 3 & 25.2              & -686.9           & 4 & 28.1              & -684.0      \\ 
		5   & 25.4            & -686.7           & 6   & 26.1            & -686.0  & 5 & 25.4              & -686.7           & 6 & 24.1              & -688.0   \\ 
		7   & 25.6            & -686.5           & 8   & 25.3            & -686.8  & 7 & 25.5              & -686.6           & 8 & 25.8              & -686.3     \\ 
	\end{tabular}
	\caption{The profile log-likelihood (PLL) and the full log-likelihood (FLL) for different link functions.}
\label{tab:tab2}
\end{table}

\begin{appendices}

This Appendix is dedicated to the proofs of the mathematical results of the paper. In Section \ref{app_A_propositions} we give the proofs of the $4$ propositions of the paper, stated in Section \ref{s3}. In these proofs we rely on some important statements, enumerated with the letter B, whose proofs are given in Section \ref{app_B_statements}. The technical results on empirical processes are all postponed to Section \ref{app_C_emp_proc}.

\counterwithin{theorem}{section}

\section{Proofs of the propositions}\label{app_A_propositions}

\subsection{Proof of Proposition \ref{proposition:p2_p3}}


In (\ref{newopti1}), write $\beta^T= (\log \theta,\gamma^T)$ with $\gamma\in \R^{d}$ and $\theta\in \R_{>0}$. Hence $\w Q_{2,\beta}(Y_i) = \w Q_\gamma(Y_i)\theta$ and (\ref{newopti1}) becomes 
\begin{align}\label{optimisationgammatheta}
 \underset{\gamma\in \R^{d},\, \theta\in\R} {\text{argmax}}\ \prod_{i=1}^n\left\{ \left(\frac{\exp(\gamma^TX_i) \theta}{\w Q_\gamma(Y_i)\theta-\w \lambda_{(\log\theta,\gamma)}}\right)^{\delta_i}  \exp( -\w\lambda_{(\log\theta,\gamma)} ) \right\}.
\end{align}
For any $\gamma\in \R^{d}$, denote by $\w \theta_\gamma$ the maximizer of (\ref{optimisationgammatheta}) over $\theta$. Hence $\w \theta_\gamma $ maximizes
\begin{align}\label{eq:maxthetagamma}
  \sum_{i=1}^n \left\{- \delta_i \log \left(\frac{\w Q_\gamma(Y_i)\theta-\w \lambda_{(\log\theta,\gamma)}}{\theta}\right) -\w\lambda_{(\log\theta,\gamma)} \right\}.
\end{align}
Furthermore, as $\w\lambda_{(\log\theta,\gamma)}$ satisfies $\sum_{i=1}^n \delta_i/(\w Q_{\gamma} (Y_i)\theta -\w\lambda_{(\log\theta,\gamma)} )=n$, a concavity argument implies that
\begin{align*}
\w\lambda_{(\log\theta,\gamma)} = \underset{{\lambda\in \mathbb R}} {\argmin}\, \sum_{i=1}^n \left\{ -\delta_i \log\left( \frac{ \w Q_{ \gamma} (Y_i) \theta - \lambda}{\theta}\right) -\lambda\right\}, 
\end{align*}
and, in particular, considering $\lambda=0$ leads to
\begin{align*}
\sum_{i=1}^n \left\{ -\delta_i \log\left( \frac{\w Q_{ \gamma} (Y_i) \theta -\w \lambda_{(\log\theta,\gamma)}}{\theta}\right) -\w \lambda_{(\log\theta,\gamma)}\right\}
 \leq   \sum_{i=1}^n -\delta_i \log\left( \w Q_{ \gamma} (Y_i) \right),
\end{align*}
for any $(\theta,\gamma)\in \mathbb R_{\geq 0}\times \mathbb R^{d} $. This inequality holds for $\theta = \w \theta_\gamma$ and it provides an upper bound for (\ref{eq:maxthetagamma}). This upper bound is achieved when $\theta$ is such that 
$\w \lambda_{(\log\theta,\gamma)}=0$, equivalently when $\w \theta_\gamma = n^{-1} \sum_{i=1}^n \delta_i/\w Q_\gamma(Y_i)$. Injecting this value in (\ref{optimisationgammatheta}) we obtain the assertion of the proposition.
\qed

\subsection{Proof of Proposition \ref{propositionconsistency}}

For every $y\in \mathbb R_{\geq 0}$, $\gamma\in  \mathbb R ^d$, write $Q_{\gamma}(y) = E[g(\gamma,X)R(y)]$.
Since the function $\gamma\mapsto E[\delta \log(g(\gamma, X)/Q_\gamma(Y) )] $ is continuous on $B$ and has a unique maximum (see the Lemma below), it suffices to show that \cite[Theorem 2.1]{newey1994}
\begin{align}\label{statement1}
\sup_{\gamma \in B} \left|n^{-1} \sum_{i=1} ^ n \delta_i\{\log(g(\gamma,X_i))-\log(\w Q_{\gamma}(Y_i))\}- E\Big[\delta\{\log(g(\gamma,X))-\log(Q_{\gamma}(Y))\}\Big]\right|  \overset{\mathbb P}{\lr} 0.\tag{B.1}
\end{align}  
This is shown in Section \ref{app_B_statements}.
\qed

\begin{lemma}\label{prop:identifiability_profile} 
\begin{enumerate}[(i)]
\item\label{lemma:consistency:_i} Under (H\ref{cond:identification1}) and (H\ref{cond:identification2}), the function $\gamma \mapsto  E[\delta\log(g(\gamma,X) / Q_\gamma (Y)) ]$ has a unique maximum $\gamma_0$. 
\item\label{lemma:consistency:_ii}  Under (H\ref{cond:consistency_gamma_0}) and (H\ref{cond:consistency_gamma_1}), the function $\gamma\mapsto E[\delta \log(g(\gamma, X)/Q_\gamma(Y) )] $ is continuous on $B$.
\end{enumerate}
\end{lemma}
\begin{proof}
We start with (\ref{lemma:consistency:_i}). Using (\ref{formula:martingale}) and that $Q_{\gamma} = E [ g(\gamma,X) R(y)] $, we get
\begin{align*}
 E\left[\delta \left(\frac{g(\gamma,X) }{ g(\gamma_0,X) } \frac{ Q_{\gamma_0} (Y)} {Q_{\gamma} (Y)} - 1 \right)  \right] = \int \left( E\left[ {g(\gamma,X) } \frac {Q_{\gamma_0} (u)}{ Q_\gamma (u)}  R(u) \right] - Q_{\gamma_0} (u) \right)d\Lambda_0(u) = 0.
\end{align*}
Since there exist $\eta,\eta'>0$ such that \citep{murphy1994}
\begin{align*}
&\log(x) - (x-1) \leq -\ell(x),\\
&\ell(x) = \eta  |x-1|\mathds 1_{\{|x-1|\geq 1/2\}} + \eta' (x-1)^2\mathds 1_{\{|x-1|< 1/2\}},
\end{align*}
it follows that
\begin{align*}
  E\left[\delta \log\left(\frac{g(\gamma,X) }{Q_{\gamma} (Y)}\right)  \right]   -   E\left[\delta \log\left( \frac { g(\gamma_0,X) } { Q_{\gamma_0} (Y)}   \right) \right] \leq -  E\left[\delta \ell \left( \frac{g(\gamma,X) }{ g(\gamma_0,X) } \frac{ Q_{\gamma_0} (Y)} {Q_{\gamma} (Y)} \right)   \right].
\end{align*}
Consequently, using (\ref{formula:martingale}), whenever $  E\left[\delta \log\left(\frac{g(\gamma,X) }{Q_{\gamma} (Y)}\right)  \right]   =    E\left[\delta \log\left( \frac { g(\gamma_0,X) } { Q_{\gamma_0} (Y)}   \right) \right] $, it holds that
\begin{align*}
\int   E  \left[ \ell \left( \frac{g(\gamma,X) }{ g(\gamma_0,X) } \frac{ Q_{\gamma_0} (u)} {Q_{\gamma} (u)}  \right)  g(\gamma_0,X)R(u) \right] d\Lambda_0 (u) = 0.
\end{align*}
For $(d\Lambda_0)$-almost every $u$, we have
\begin{align*}
 E \left[ \ell \left( \frac{g(\gamma,X) }{ g(\gamma_0,X) } \frac{ Q_{\gamma_0} (u)} {Q_{\gamma} (u)} \right)  g(\gamma_0,X)R(u) \right] = 0. 
\end{align*}
But by (H\ref{cond:identification1}), it holds, almost surely, $ \inf_{u\in [0,\tau]}  g(\gamma_0,X)  E [ R(u)|X] =  g(\gamma_0,X)   E [ R(\tau )|X] >0 $, which implies that
almost surely $({g(\gamma,X) }/{ g(\gamma_0,X) })({ Q_{\gamma_0} (u)} /{Q_{\gamma} (u)}) = 1 $.\\ Hence ${g(\gamma,X) }/{ g(\gamma_0,X) } $ is constant and we conclude using (H\ref{cond:identification2}).

We continue with (\ref{lemma:consistency:_ii}). We proceed in two parts. We first consider the function $\gamma\mapsto E[\delta \log(g(\gamma, X))] $ and second we deal with $\gamma\mapsto E[\delta \log(Q_\gamma(Y) )] $. Because $\delta$ is bounded, it suffices to show that
$\int \left| \log(g(\gamma, x)/g(\tilde \gamma , x))\right|  dP(x)\overset{\gamma \rightarrow \tilde \gamma}{ \lr } 0$.
We apply the Lebesgue dominated convergence theorem. For every $x\in \mathcal S$, the continuity of the function $g(\gamma, x)$ at $\tilde \gamma$ and the fact that $g(\gamma, x)$ is bounded from below implies that $| \log(g(\gamma, x)/g(\tilde \gamma , x))| \rightarrow 0 $ whenever $\gamma\rightarrow \tilde \gamma$. By (H\ref{cond:consistency_gamma_1}), we also have that
\begin{align*}
| \log(g(\gamma, x)/g(\tilde \gamma , x))|  \leq |\log(m_1(x))| + |\log( M_1(x) ) |,
\end{align*}
which is $(dP)$-integrable. To obtain that
$\int \left| \log(Q_\gamma(y)/Q_{\tilde \gamma}(y))\right|  dP(y)\overset{\gamma \rightarrow \tilde \gamma}{ \lr } 0$,
we can follow the same path as before by applying the Lebesgue dominated convergence theorem. We have that, for every $\gamma\in B$,
\begin{align}\label{eq:bound_upper_lower}
E[m_1(X)  (1-\Delta)] \leq Q_\gamma (y ) \leq E[M_1(X)].
\end{align} 
The continuity of the function $g(\gamma, x)$ at $\tilde \gamma$ implies the continuity of $\gamma \rightarrow Q_\gamma(y)$ for every $y\in\mathbb R_{\geq 0}$ (by another application of the Lebesgue dominated convergence theorem), which gives, with the help of (\ref{eq:bound_upper_lower}), that $| \log(Q_\gamma( y)/Q_{\tilde \gamma}(y))| \rightarrow 0 $ whenever $\gamma\rightarrow \tilde \gamma$. It remains to note that
$| \log(Q_\gamma( y)/Q_{\tilde \gamma}(y))|\leq |\log(E[m_1(X) (1-\Delta)])| + |\log( E[M_1(X)] ) |<+\infty$.
\end{proof}

\subsection{Proof of Proposition \ref{proposition:weak_cv_gamma}}
For every $\gamma\in B$, define $\w h_\gamma(y)= {\nabla_{\gamma} \w Q_\gamma(y)}/{\w Q_\gamma(y)}$.
It is worth mentioning that, for every $\gamma\in B$,
\begin{align}\label{eq:importantidentity}
n^{-1}\sum_{i=1} ^n \int d_\gamma (X_i)  g(\gamma,X_i) R_i(u) \, d\Lambda_0(u) &= \int \nabla_\gamma \w Q_\gamma(u)d\Lambda_0(u) \nonumber \\
&= \int \w h_\gamma(u) \w Q_\gamma(u)d\Lambda_0(u) \nonumber\\
&=n^{-1}\sum_{i=1} ^n  \int \w h_\gamma(u) g(\gamma,X_i)R_i(u) d\Lambda_0(u).
\end{align}
As by (H\ref{cond:asymptotics_gamma_2}), $\gamma \mapsto g(\gamma, x) $ is differentiable, $\w\gamma$ satisfies the equation $S_n(\w \gamma)= 0$, where
\begin{align*}
S_n (\gamma) = n^{-1}\sum_{i=1}^n \int \{d_\gamma (X_i) - \w h_\gamma(u) \} dN_i(u) .
\end{align*}
We rely on the following decomposition. Based on (\ref{eq:importantidentity}), for every $\gamma\in B$,
\begin{align*}
S_n (\gamma) &=  n^{-1}\sum_{i=1}^n \int \{d_\gamma (X_i) - \w h_\gamma(u)\} (dN_i(u) -g(\gamma, X_i ) R_i(u) d\Lambda_0(u)) \\
&=  n^{-1}\sum_{i=1}^n \int \{d_\gamma (X_i) - \w h_\gamma(u)\} dM_i(u) \\
&\qquad\qquad +  n^{-1}\sum_{i=1}^n \int \{d_\gamma (X_i) - \w h_\gamma(u)\}(g(\gamma_0, X_i ) -g(\gamma, X_i ))  R_i(u) d\Lambda_0(u).
\end{align*}
Applying this for $\gamma\in B$ and for $\gamma_0$ implies that
\begin{align*}
S_n(\gamma) - S_n(\gamma_0) 
&= n^{-1}\sum_{i=1}^n \int \{d_\gamma (X_i) - \w h_\gamma(u)\}(g(\gamma_0, X_i ) -g(\gamma, X_i ))  R_i(u) d\Lambda_0(u) + r_{1,n}(\gamma) ,
\end{align*}
with 
$r_{1,n}(\gamma) = n^{-1}\sum_{i=1}^n \int \{d_\gamma (X_i)-d_0(X_i) +\w h_0(u) - \w h_\gamma(u)\} dM_i(u)$.
Taking $\gamma = \w \gamma$, for which $S_n(\w \gamma)=0$, and using the mean-value theorem around the value $\gamma_0$ with the map
\begin{align*}
 \gamma \mapsto n^{-1}\sum_{i=1}^n \int \{d_{\widehat\gamma} (X_i) - \w h_{\widehat \gamma}(u)\} g( \gamma, X_i )  R_i(u) d\Lambda_0(u),
\end{align*}
which is continuously differentiable by (H\ref{cond:asymptotics_gamma_2}), gives
$-S_n(\gamma_0) =  - H_n (\tilde \gamma) (\w \gamma-\gamma_0)+r_{1,n}(\widehat \gamma)$,
with $\tilde \gamma$ on the line segment between $ \w\gamma$ and  $\gamma_0$, and 
\begin{align*}
 H_n (\tilde \gamma) = n^{-1}\sum_{i=1}^n \int \{d_{\widehat\gamma} (X_i) - \w h_{\widehat \gamma}(u)\}\nabla_\gamma g(\tilde \gamma, X_i )^T   R_i(u) d\Lambda_0(u)  .
\end{align*}
We show in Section \ref{app_B_statements} that 
\begin{align}\label{statement2}
&H_n (\tilde \gamma) \overset{\mathbb P}{\lr} I_0,\tag{B.2}\\
&n^{1/2} r_{1,n}(\widehat \gamma) \overset{\mathbb P}{\lr}0 .\label{statement2.0}\tag{B.3}
\end{align}
Because the matrix $I_0$ has full rank by (H\ref{cond:asymptotics_gamma_1}), we know from (\ref{statement2}) that with probability tending to $1$, $H_n (\tilde \gamma)$ is invertible. Then using (\ref{statement2.0}) gives that
$n^{1/2} (\w \gamma-\gamma_0) =  H_n ( \tilde \gamma)^{-1} \{ n^{1/2}  S_n(\gamma_0)\} + o_{\mathbb P}(1)$,
hence it remains to show that (see Section \ref{app_B_statements})
\begin{align}\label{statement2prime}
n^{1/2} S_n(\gamma_0) = n^{-1/2}\sum_{i=1} ^n \int (d_0(X_i) -  h_0(u))dM_i(u) + o_{\mathbb P}(1),\tag{B.4}
\end{align}
to deduce the statement.  

\qed

\subsection{Proof of Proposition \ref{proposition:strong_decomp_Lambda}}
For every $y\in \mathbb R_{\geq 0}$, write
\begin{eqnarray*}
n^{1/2}(\w \Lambda(y)-\Lambda_0(y))
&=&
n^{-1/2} \sum_{i=1}^n \int_0^y \frac{1}{\w Q_{\w \gamma}(u)} (dN_i(u)-\w Q_{\w \gamma}(u)d\Lambda_0(u)) \\
&=& n^{-1/2} \sum_{i=1}^n \int_0^y \frac{dM_i(u)}{\w Q_{\w \gamma}(u)} +   n^{1/2} \int_0^y \left(\frac{ \w Q_{ \gamma_0}(u) -  \w Q_{\w \gamma}(u) }{\w Q_{\w \gamma}(u)} \right) d\Lambda_0(u) \\
&=& n^{-1/2} \sum_{i=1}^n \int_0^y \frac{dM_i(u)}{\w Q_{\w \gamma}(u)} -  n^{1/2}  (\w \gamma-\gamma_0) \int_0^y \left(\frac{  \nabla_\gamma \w Q_{ \tilde \gamma}(u) }{\w Q_{\w \gamma}(u)} \right) d\Lambda_0(u),
\end{eqnarray*}
 for some $\tilde \gamma$ belonging to the line segment between $ \w\gamma$ and  $\gamma_0$.
As shown in Section \ref{app_B_statements},
\begin{align}\label{statement3}
&\sup_{y\in \mathbb R_{\geq 0} } \left| n^{-1/2} \sum_{i=1}^n \int_0^y \left( \frac{1}{\w Q_{\w \gamma}(u)} -\frac{1}{ Q_{0}(u)}  \right) dM_i(u)\right| =o_{\mathbb P}(1),\tag{B.5}
\end{align}
and since, from Lemma \ref{Lemma:probacv},
\begin{align*}
& \sup_{u\in \R_{\geq 0}}   \left|\frac{  \nabla_\gamma \w Q_{ \tilde \gamma}(u) }{\w Q_{\w \gamma}(u)}- h_0(u) \right|     = o_{\mathbb P}(1),
\end{align*}
the result follows.
\qed

\section{Proof of the auxiliary statements (\ref{statement1}) to (\ref{statement3})} \label{app_B_statements}

\paragraph{Proof of (\ref{statement1}):}
First, we deal with the terms of the form $ \delta  \log(g(\gamma,x)) $. From Lemma \ref{Lemma:GCclass} assertion (\ref{GC:consistency1}), the underlying class indexed by $\gamma \in B$, is Glivenko-Cantelli. It follows that
\begin{align*}
\sup_{\gamma \in B}\left |n^{-1} \sum_{i=1} ^ n \delta_i \log(g(\gamma,X_i))- E[\delta\log(g(\gamma,X))] \right|  \overset{\mathbb P}{\lr} 0.
\end{align*}
Second, 
with probability going to $1$, we have that (with $ b =  E[m_1(X)(1-\Delta) ]/2$)
\begin{align*}
&\sup_{\gamma \in B} \left |n^{-1} \sum_{i=1} ^ n \delta_i \log(\w Q_{\gamma}(Y_i))- E[\delta\log(Q_{\gamma}(Y))] \right|  \\
&\leq \sup_{\gamma \in B} \left |n^{-1} \sum_{i=1} ^ n\delta_i\{ \log(\w Q_{\gamma}(Y_i))- \log(Q_{\gamma}(Y_i))\}\right | +\sup_{\gamma \in B}\left |n^{-1} \sum_{i=1} ^ n \delta_i\log( Q_{\gamma}(Y_i))- E[\delta\log(Q_{\gamma}(Y))]\right | \\
&\leq  2 b^{-1} \sup_{\gamma \in B,\, y\in \R_{\geq 0} }\left | \w Q_{\gamma}(y)- Q_{\gamma}(y)\right | +\sup_{\gamma \in B}\left |n^{-1} \sum_{i=1} ^ n\delta_i \log( Q_{\gamma}(Y_i))- E[\delta\log(Q_{\gamma}(Y))]\right |,
\end{align*}  
which follows from the mean-value theorem applied to $x\mapsto \log(x)$, and where the bound, in probability, is given in (\ref{convergence:12}) of Lemma \ref{Lemma:probacv}. Convergence of the first term above is then implied by Lemma \ref{Lemma:probacv}, equation (\ref{convergence:11}). Convergence of the second term above is deduced from Lemma \ref{Lemma:GCclass}, assertion (\ref{GC:consistency2}).  \qed

\paragraph{Proof of (\ref{statement2}):}
We show that for any random sequences $\gamma_n$ and $\tilde \gamma_n$ going to $ \gamma_0$, in $\mathbb P$-probability, we have
\begin{align*}
 n^{-1}\sum_{i=1}^n \int \{d_{\gamma_n} (X_i) - \w h_{\gamma_n}(u)\}\nabla_\gamma g(\tilde \gamma_n, X_i )^T   R_i(u) d\Lambda_0(u)  \overset{\mathbb P}{\lr} I_0.
\end{align*}
Some basic algebra implies that, for any bounded function $h$, 
\begin{align*}
\int  E \left[ \{d_0 (X) -  h_0(u)\} h(u )  g(\gamma_0,X)  R(u)\right] d\Lambda_0(u)=0.
\end{align*}
From the previous with $h=h_0$, we deduce that
\begin{align*}
I_0 =  \int E\left[  \{d_{0} (X) - h_{0}(u)\} d_{0} (X)^T   g(\gamma_0, X)   R(u) \right] d\Lambda_0(u) ,
\end{align*}
hence, we have to prove that
\begin{align*}
 \int  n^{-1}\sum_{i=1}^n  \left [ \{d_{\gamma_n} (X_i) - \w h_{\gamma_n}(u)\} \nabla_\gamma g(\tilde \gamma_n, X_i )^T   R_i(u)\right] d\Lambda_0(u)
 \\
   \overset{\mathbb P}{\lr} \int E\left[  \{d_{0} (X) - h_{0}(u)\} \nabla_\gamma g(\gamma_0, X)^T   R(u) \right] d\Lambda_0(u) .
\end{align*}
From the triangle inequality, defining $ a(Y_i)= \int R_i(u ) d\Lambda_0(u) $, it is enough to prove that
\begin{align*}
& \left| n^{-1}\sum_{i=1}^n  d_{\gamma_n} (X_i)  \nabla_\gamma g(\tilde \gamma_n, X_i )^T   a(Y_i) - E\left[  d_{0} (X)  \nabla_\gamma g(\gamma_0, X )^T   a(Y) \right] \right|  \overset{\mathbb P}{\lr} 0,\\
&\sup_{y\in\R_{\geq 0}} \left|      \w h_{\gamma_n}(y) \nabla_\gamma \w  Q_{\tilde \gamma_n} (y)^T -    h_{0}(y) \nabla_\gamma Q_{0} (y)^T    \right|  \overset{\mathbb P}{\lr} 0.
\end{align*}
From Lemma \ref{Lemma:GCclass}, the functions $(x,y) \mapsto d_{\gamma} (x) \nabla_\gamma g(\tilde \gamma, x )^T   a(y)$, with $\gamma$ and $\tilde \gamma$ in $B$, are included in a Glivenko-Cantelli class. Hence,
\begin{align*}
\sup_{\gamma\in B,\, \tilde \gamma\in B} \left| n^{-1}\sum_{i=1}^n  d_{\gamma} (X_i)  \nabla_\gamma g(\tilde \gamma, X_i )^T   a(Y_i) - E\left[  d_{\gamma} (X)  \nabla_\gamma g(\tilde \gamma, X)^T   a(Y)\right] \right|  \overset{\mathbb P}{\lr} 0.
\end{align*} 
Hence, the first convergence is derived from the continuity of the map 
$(\gamma, \tilde \gamma) \mapsto E[  d_{\gamma} (X) \linebreak  \nabla_\gamma g(\tilde \gamma, X )^T   a(Y)]$.  
This is implied by (H\ref{cond:consistency_gamma_1}) and (H\ref{cond:asymptotics_gamma_2}) invoking the continuity of $\gamma \mapsto  g( \gamma, x)$ and $\gamma \mapsto \nabla_\gamma g( \gamma, x)$, for every  $x\in \mathcal S$ and the Lebesgue dominated convergence theorem. The second convergence is a direct consequence of Lemma \ref{Lemma:probacv}, (\ref{convergence:12}), (\ref{convergence:13}) and (\ref{convergence:22}). 
\qed

\paragraph{Proof of (\ref{statement2.0}):}
We proceed in two steps. First, we show that, for any sequence $\gamma_n$ going to $0$, in $\mathbb P$-probability,
\begin{align*}
 n^{-1}\sum_{i=1}^n  \int \{d_{\gamma_n} (X_i)-d_0(X_i)\} dM_i(u)  =o_{\mathbb P}(n^{-1/2}).
\end{align*}
We apply Lemma \ref{Lemma:equicontinuity1} to obtain the previous convergence coordinate by coordinate. For $j\in\{1,\ldots, q\}$, with probability $1$, the function $d_{\w \gamma,j}$ belongs to the class $\{x\mapsto d_{\gamma,j} (x) \,:\, \gamma\in B\}$ which, by Lemma \ref{lemma:dgamma_donsker&hboundedvariation}, satisfies (\ref{cond:uniform_entropy}). By (H\ref{cond:consistency_gamma_0}) and (H\ref{cond:asymptotics_gamma_2}), there exists some constant $C>0$ such that the envelop $L$, given in Lemma \ref{lemma:dgamma_donsker&hboundedvariation}, satisfies
\begin{align*}
&E[L^2(X) g(\gamma_0,X) ] \\
&< C \left( E\left[ \frac{(c_2(X)+M_2(X) )^2 M_1(X)}{m_1^2(X)}   \right]   +E\left[ \frac{M_2^2(X) (c_1(X)+ M_1(X))^2M_1(X)}{m_1^4(X)}    \right]
 \right)<+\infty .
\end{align*}
 Moreover from (\ref{ineq:dgamma_lipschitz}) and (\ref{eq:mean_value_th2}), we find that
$E[\{d_{ \gamma_n,j} (X)-d_{\gamma_0,j}(X)\} g(\gamma_0,X)]\leq c_1 | \gamma_n-\gamma_0|$,
which goes to $0$ in $\mathbb P$-probability.

Second, we prove that, for any sequence $\gamma_n$ going to $0$, in $\mathbb P$-probability,
\begin{align}\label{statement2.0.1}
 n^{-1}\sum_{i=1}^n \int \left\{\w h_{0}(u) - \w h_{ \gamma_n}(u)\right\}  d M_i(u)=o_{\mathbb P}(n^{-1/2}).
\end{align}
Let $j\in\{1,\ldots, q\}$.  Then, by Lemma \ref{lemma:dgamma_donsker&hboundedvariation}, $\widehat h_{\gamma_n,j} \in \text{BV} (m,v) $ with probability going to $1$, and by Lemma \ref{Lemma:probacv}, $\sup_{u\in \mathbb R_{\geq 0}} |h_{\gamma_0,j}(u) - \w h_{\gamma_n,j}(u)|\overset{\mathbb P}{\rightarrow }0$.  Hence, the result follows from Lemma \ref{Lemma:equicontinuity2}.
\qed

\paragraph{Proof of (\ref{statement2prime}):}
From identity (\ref{eq:importantidentity}) with $\gamma = \gamma_0$, we have
\begin{align*}
S_n(\gamma_0) &= n^{-1}\sum_{i=1} ^n \int (d_0(X_i) - \w h_0(u))dN_i(u)\\
&= n^{-1}\sum_{i=1} ^n \int ( d_0(X_i) - \w h_0(u))dM_i(u).
\end{align*} 
Using (\ref{statement2.0.1}) with $\gamma_n = \gamma_0$, gives
$n^{-1}\sum_{i=1}^n \int \left\{h_0(u) - \w h_{0}(u)\right\}  d M_i(u)=o_{\mathbb P}(n^{-1/2})$,
and hence (\ref{statement2prime}) follows.
\qed

\paragraph{Proof of (\ref{statement3}):} We will apply Lemma \ref{Lemma:equicontinuity2} with $\widehat h$ equal to the function $u\mapsto  \w Q_{\w \gamma}(u)^{-1} $ and $h_0$ equal to the function $u\mapsto   Q_{0}(u)^{-1} $. By (\ref{convergence:12}), the functions $\{\w Q_{\gamma}^{-1} : \gamma\in B\}$, are, with probability going to $1$, valued in a bounded interval. The fact that they are non-decreasing implies that their total variation is smaller than $|2/E[m_1(X)(1-\Delta)] - 1/(2E[M_1(X)])|$, with probability going to $1$. It follows that there exist $m$ and $v$ such that with probability going to $1$, $\{ \w Q_{\gamma}^{-1}\, :\, \gamma\in B\}\subset \text{BV}(m,v) $.   Furthermore, on the event $\{ \inf_{\gamma \in B,\, y\in \R_{\geq 0}} \w Q _\gamma (y)\geq mE(1-\Delta)\} $, which has probability going to $1$ in light of Lemma \ref{Lemma:probacv}, equation (\ref{convergence:12}), we have
\begin{align*}
\sup_{u\in \mathbb R_{\geq 0}} |\w Q_{\w \gamma}(u)^{-1} - Q_{0}(u)^{-1}|\leq  4  E[m_1(X) (1-\Delta)] ^{-2} \sup_{u\in \mathbb R_{\geq 0}} |\w Q_{\w \gamma}(u) - Q_{0}(u)| \overset{\mathbb P}{\longrightarrow }0.
\end{align*}
\qed

\section{Technical lemmas on empirical processes}\label{app_C_emp_proc}

Empirical process theory is useful to describe the asymptotics of semiparametric estimators because they usually result in empirical sums indexed possibly by some functional quantities. Helpful concepts are Glivenko-Cantelli classes and Donsker classes, as studied in \cite{vandervaart1996}. We start by showing the Glivenko-Cantelli property for certain classes of interest. Let $\xi,\xi_1,\xi_2,\ldots$ be independent and identically distributed random variables with distribution $P$. Denote the underlying probability by $\mathbb P$. A class $\mathcal F$ of real-valued functions is said to be $P$-Glivenko-Cantelli if
\begin{align*}
\sup_{f\in \mathcal F} \left| n^{-1}\sum_{i=1}^n \{ f(\xi_i)- Ef(\xi)\}\right| \overset{\mathbb P}{\longrightarrow} 0.
\end{align*}
When $\mathcal F$ is a vector-valued class, we say it is $P$-Glivenko-Cantelli when each coordinate is $P$-Glivenko-Cantelli. In what follows, the $j$-th coordinate of $d_{\gamma} $ and $\nabla _{\gamma} \widehat Q_\gamma $ are denoted by $d_{\gamma,j} $ and $\nabla _{\gamma,j} \widehat Q_\gamma  $, respectively ($j= 1,\ldots, q$).

\begin{lemma}\label{Lemma:GCclass}
Let $R_{y}(u) =  \mathds 1_{\{ y\leq \tau \}} \mathds 1_{\{ y\geq u\}} + \mathds 1_{\{y>\tau \}} $.  Under (H\ref{cond:consistency_gamma_0}) and (H\ref{cond:consistency_gamma_1}), the following holds:
\begin{enumerate}[(i)]
\item\label{GC:consistency1}   the class $\left\{ (\delta,x) \mapsto   \delta \log(g(\gamma,x) ) \,:\, \gamma\in B \right\}$ is $P$-Glivenko-Cantelli,
\item \label{GC:consistency2}   the class $\left\{ (\delta, y)  \mapsto \delta\log( Q_\gamma(y))   \,:\,  \gamma\in B \right\}$ is  $P$-Glivenko-Cantelli,
\item \label{GC3} the class $\left\{(x,y) \mapsto  g(\gamma,x)R_y(u)  \,:\,  \gamma\in B,\, u\in \mathbb R_{\geq 0} \right\}$ is  $P$-Glivenko-Cantelli. 
\end{enumerate}
Let $ a(y)= \int R_y(u ) d\Lambda_0(u) $. Under (H\ref{cond:consistency_gamma_0}), (H\ref{cond:consistency_gamma_1}) and (H\ref{cond:asymptotics_gamma_2}) the following holds for all $j,k\in \{1,\ldots, q\}$: 
\begin{enumerate}[(i)]  \setcounter{enumi}{3}
\item\label{GC:prime1}  the class $\left\{(x,y) \mapsto \nabla_{\gamma,j} g(\gamma,x) R_y(u)   \,:\,  \gamma\in B,\, u\in \mathbb R_{\geq 0} \right\}$ is $P$-Glivenko-Cantelli,
\item\label{GC:prime1.1}  the class $\left\{(x,y) \mapsto  | \nabla_{\gamma,j}g(\gamma,x) | R_y(u)   \,:\,  \gamma\in B,\, u\in \mathbb R_{\geq 0} \right\}$ is $P$-Glivenko-Cantelli,
\item \label{GC:prime2} the class $\left\{ (x,y) \mapsto d_{\gamma,j} (x) \nabla_{\gamma,k} g(\tilde \gamma, x )^T   a(y) \, : \, \gamma\in B,\, \tilde \gamma\in B \right\}$ is $P$-Glivenko-Cantelli.
\end{enumerate}

\end{lemma}

\begin{proof}
Let $\mathcal N_{[\,]}(\epsilon,\mathcal F, \|\cdot\|)$ (resp. $\mathcal N(\epsilon,\mathcal F, \|\cdot\|)$) denote the $\epsilon$-bracketing number (resp. $\epsilon$-covering number) of the metric space $(\mathcal F,\|\cdot\|)$ \citep[Definition 2.1.5]{vandervaart1996}.

As a preliminary step, we show that the class $\mathcal G=\left\{ x \mapsto   g(\gamma,x ) \,:\, \gamma\in B \right\}$ is Glivenko-Cantelli whenever $0< E[c_1(X)]<+\infty$ which is true by (H\ref{cond:consistency_gamma_1}). 
Because of (\ref{eq:mean_value_th1}), we are in position to apply Theorem 2.7.11 in \cite{vandervaart1996} with the $L_p(Q)$-norm, $p\geq 1$, $Q$ some probability measure, and the class $\mathcal G$. Let $B_0$ be some ball of finite radius in $\mathbb R^{q}$ such that $B\subset B_0$. Because the $\epsilon$-covering number of $(B_0,|\cdot|_1)$ is $O(\epsilon^ {-q})$, we find 
\begin{align}\label{ineq:bracketing_G}
\mathcal N_{[\,]}\left(  2 \epsilon  \|  c_1\|_{L_p(Q)},\mathcal G, L_p(Q)\right)  \leq  \mathcal N_{}(\epsilon,B_0, |\cdot|_1) =  K\epsilon^{-q},
\end{align}
for some $K>0$. When $p=1$ and $Q=P$, because $0< \|  c_1\|_{L_1(P)}< +\infty  $, we have that $\mathcal N_{[\,]}\left(   \epsilon  ,\mathcal G, L_1(P)\right)<+\infty $ for every $\epsilon >0$, making the class $\mathcal G$ is Glivenko-Cantelli \citep[Theorem 2.4.1]{vandervaart1996}. 
 

We now prove (i). The class of interest $\{ (\delta, x ) \mapsto \delta \log(g(\gamma,x) ) \, :\, \gamma\in B\}$ can be written as $\mathcal F_1\times \log(\mathcal F_2)$ where $\mathcal F_1 = \{\delta\mapsto \delta\}$ and $\mathcal F_2 = \{  x  \mapsto \log(g(\gamma,x) ) \, :\, \gamma\in B\}$. This is a continuous transformation of a Glivenko-Cantelli class and we can apply Theorem 3 in \cite{vandervaart+w:2000}. The envelop property is ensured as, by (H\ref{cond:consistency_gamma_1}), $E [\sup_{\gamma\in B} | \log(g(\gamma,X))| ]\leq E [ |\log(M_1(X))|  +|\log(m_1(X))|] <+\infty $.

We now consider (ii). Multiplying (\ref{eq:mean_value_th1}) by $R_Y(y)$ and taking the expectation, we obtain, for every $y\in\mathbb R_{\geq 0}$, $(\gamma,\tilde \gamma)\in B^2$,
\begin{align}\label{eq:bound_lipschitz_Q}
|  Q_\gamma(y)  -   Q_{\tilde \gamma}(y)  | \leq |\gamma - \tilde \gamma |_1  E[ c_1(X)].
\end{align}
Following the preliminary step of the proof with (\ref{eq:bound_lipschitz_Q}) in place of (\ref{eq:mean_value_th1}), we again invoke Theorem 2.7.11 and Theorem 2.4.1 in \cite{vandervaart1996} to obtain that $\left\{  y  \mapsto  Q_\gamma(y)   \,:\,  \gamma\in B \right\}$ is Glivenko-Cantelli.
Then applying Theorem 3 in \cite{vandervaart+w:2000}, we show (\ref{GC:consistency2}). The (constant) envelop property is provided by (\ref{eq:bound_upper_lower}).

To show  (\ref{GC3}), we apply Theorem 3 in \cite{vandervaart+w:2000} as the class of interest is the product of two classes, $\mathcal G$ and $\{y\mapsto R_y(u)\, :\, u\in \mathbb R_{\geq 0} \}$, both being $P$-Glivenko-Cantelli.

For (\ref{GC:prime1}), let $j\in \{1,\ldots, q\}$. Similarly to the preliminary step, we show that $\{\nabla _{\gamma,j} g(\gamma, x) \,:\, \gamma\in B\} $ is $P$-Glivenko-Cantelli provided that $0<E[c_2(X)] <+\infty $. Then as when proving (\ref{GC3}), because $E[ M_2(X)]<+\infty $ we apply Theorem 3 in \cite{vandervaart+w:2000} to obtain (\ref{GC:prime1}).

For (\ref{GC:prime1.1}),  Theorem 3 in \cite{vandervaart+w:2000} applied to  $\big\{(x,y) \mapsto \nabla_{\gamma,j} g(\gamma,x) R_y(u) : $  $ \gamma\in B,\, u\in \mathbb R_{\geq 0} \big\}$ gives that
$\left\{(x,y) \mapsto |\nabla_{\gamma,j} g(\gamma,x)| R_y(u)   \,:\,  \gamma\in B,\, u\in \mathbb R_{\geq 0} \right\}$ is $P$-Glivenko-Cantelli.

Concerning (\ref{GC:prime2}), let $j,k\in \{1,\ldots, q\}$. The class of interest is a continuous transformation of the $P$-Glivenko-Cantelli classes,
\begin{align*}
&\left\{(x,y) \mapsto \nabla_{\gamma,j} g(\gamma,x)    \,:\,  \gamma\in B,\, u\in\mathbb R_{\geq 0} \right\},\quad  \left\{(x,y) \mapsto \nabla_{\gamma,k} g(\gamma,x)    \,:\,  \gamma\in B,\,u\in\mathbb R_{\geq 0} \right\}\\
&\left\{x \mapsto   g(\gamma,x)    \,:\,  \gamma\in B \right\},\quad \{y\mapsto a(y) \}.
\end{align*}
Consequently, one just has to verify the envelop property which is obtained from (H\ref{cond:consistency_gamma_1}) and (H\ref{cond:asymptotics_gamma_2}),
\begin{align*}
E \left[\sup_{\gamma\in B, \, \tilde \gamma\in B} \left| d_{\gamma,j} (X) \nabla_{\gamma,k} g(\tilde \gamma, X )^T a(Y) \right| \right]< \theta_0 E\left[\frac{ M_2^2(X) }{ m_1(X) } \right] <+\infty.
\end{align*}
\end{proof}

\begin{lemma}\label{Lemma:probacv}
Let $ \gamma_n$ be a random sequence that converges to $\gamma_0$ in $\mathbb P$-probability. Under (H\ref{cond:consistency_gamma_1}), we have that
\begin{align}
&  \sup_{\gamma \in B,\, y\in \R_{\geq 0} } | \w Q_{\gamma}(y)- Q_{\gamma}(y)|\overset{\mathbb P}{\lr} 0,\label{convergence:11}\\
 &\mathbb P \left( \forall\gamma \in B,\, \forall y\in \R_{\geq 0}\,:\,  E[ m_1(X) (1-\Delta)]/2 \leq  \widehat Q_\gamma (y )  \leq 2E[M_1(X)]  \right)  \lr 1, \label{convergence:12}\\
&\sup_{y\in\R_{\geq 0}} \left|     \w  Q_{ \gamma_n} (y) - Q_{0} (y)   \right|  \overset{\mathbb P}{\lr} 0.\label{convergence:13} 
\end{align}
Under (H\ref{cond:consistency_gamma_1}) and (H\ref{cond:asymptotics_gamma_2}), we have that
\begin{align}
&  \sup_{\gamma \in B,\, y\in \R_{\geq 0} } \left| \nabla_\gamma \w Q_{\gamma}(y)- \nabla_\gamma Q_{\gamma}(y)\right|_1\overset{\mathbb P }{\lr} 0, \label{convergence:21} \\
 &\mathbb P \left( \sup_{\gamma \in B,\, y\in \R_{\geq 0}} \left| \nabla_\gamma \w Q_{\gamma}(y) \right|_1 \leq 2E[M_2(X)]  \right)  \lr 1, \label{convergence:21.1}\\
&\sup_{y\in\R_{\geq 0}} \left|   \nabla_\gamma \w  Q_{ \gamma_n} (y) -     \nabla_\gamma Q_{0} (y)    \right| _1 \overset{\mathbb P}{\lr} 0\label{convergence:22} .
\end{align}
\end{lemma}

\begin{proof}
Convergences (\ref{convergence:11}) and (\ref{convergence:21}) are consequences of, respectively, (\ref{GC3}) and (\ref{GC:prime1}) of Lemma \ref{Lemma:Donskerclass}. Statement (\ref{convergence:12})  is an easy consequence of (\ref{eq:bound_upper_lower})  and (\ref{convergence:11}). Similarly, we obtain (\ref{convergence:21.1}) invoking (\ref{convergence:21}) and the fact that, from (H\ref{cond:asymptotics_gamma_2}),
$\left| \nabla_\gamma Q_{\gamma}(y) \right|_1\leq E[ M_2(X)]$.
Convergences (\ref{convergence:13}) and (\ref{convergence:22}) are treated similarly. Indeed, for (\ref{convergence:13}), write
\begin{align*}
\sup_{y\in\R_{\geq 0}} \left|     \w  Q_{ \gamma_n} (y) - Q_{0} (y)   \right|  \leq \sup_{y\in\R_{\geq 0},\, \gamma\in B} \left|   \w  Q_{ \gamma } (y) -  Q_{\gamma} (y)   \right| +   \sup_{y\in\R_{\geq 0}} \left|      Q_{ \gamma_n} (y) - Q_{0} (y)   \right|.
\end{align*}
The first term on the right hand side goes to $0$ in $\mathbb P$-probability as shown before. For the second term on the right hand side, (\ref{eq:bound_lipschitz_Q}) yields
\begin{align*}
 \sup_{y\in\R_{\geq 0}} \left|      Q_{ \gamma_n} (y) - Q_{0} (y)   \right|\leq |\gamma_n-\gamma_0 |_1  E[ c_1(X)] .
\end{align*}
The conclusion follows. For (\ref{convergence:22}) we do the same and obtain from (\ref{eq:mean_value_th2}) that
\begin{align*}
 \sup_{y\in\R_{\geq 0}} \left|      \nabla_\gamma Q_{ \gamma_n} (y) - \nabla_\gamma Q_{0} (y)   \right|_1\leq |\gamma_n-\gamma_0 |_1  E[ c_2(X)] .
\end{align*}
The result now follows.
\end{proof}

We now turn our attention to some results related to the concept of Donsker classes. A class $\mathcal F$ is said to be $P$-Donsker if
\begin{align*}
 n^{1/2} \left( n^{-1}\sum_{i=1}^n \{ f(\xi_i)- Ef(\xi)\}\right) \text{converges weakly to a Gaussian process in the space }\ell ^{\infty}(\mathcal F).
\end{align*}
The space $\ell ^{\infty}(\mathcal F)$ denotes the metric space of bounded functions defined on $\mathcal F$ endowed with the supremum distance. Let $\text{BV}(m,v)$ denote the space of c\`ad-l\`ag functions bounded by $m$ and with bounded variation $v$. Define, for every $(y,u,x)\in \mathbb R_{\geq 0}\times \mathbb R_{\geq 0}\times \mathcal S$,
\begin{align*}
M_{y,\delta, x}(u )  = \delta \mathds 1_{\{y\leq u\}} - \int_0^u g(\gamma_0,x) (\mathds 1_{\{ y\leq \tau \}} \mathds 1_{\{ y\geq v \}} + \mathds 1_{\{y>\tau \}} ) d\Lambda_0(v).
\end{align*}

\begin{lemma}\label{Lemma:Donskerclass}
Suppose that $E[g(\gamma_0,X)^2]<+\infty$. The class $ \{ (y,\delta, x) \mapsto  \int h(u) dM_{y,\delta, x}(u)\, : \, h \in \text{BV} (m,v)  \}$ is $P$-Donsker. 
\end{lemma}

\begin{proof}

As a first step, we show that $\{n^{-1/2} \sum_{i=1}^n M_i(u)\, :\, u\in\mathbb R_{\geq 0}\}   $ converges weakly  in $\ell^\infty (\mathbb R_{\geq 0})$. 
Example 2.5.4 in \cite{vandervaart1996} provides a bound on the uniform entropy numbers of the class of indicator functions. Example 2.10.23 in \cite{vandervaart1996} ensures that the product of two such classes is Donsker. It implies that $\{(y,\delta)\mapsto \delta \mathds 1 _{\{ y \leq u\}}\, : \, u\in\mathbb R\}$ is Donsker. Moreover, the set of functions defined for any  $ (y,x)\in \mathbb R_{\geq 0}\times \mathcal S $ by
\begin{align*}
\int_0^u g(\gamma_0,x) (\mathds 1_{\{ y\leq \tau \}} \mathds 1_{\{ y\geq v \}} + \mathds 1_{\{y>\tau \}} ) d\Lambda_0(v) =  g(\gamma_0,x)\left (\mathds 1_{\{ y\leq \tau \}} \Lambda_0 (y\wedge u ) + \mathds 1_{\{ y> \tau \}} \Lambda_0 ( u) \right),
\end{align*}  
when $u$ varies in $\mathbb R_{\geq 0}$, is VC. Indeed, the class $\{y\mapsto \mathds 1_{\{ y\leq \tau \}} \Lambda_0 (y\wedge u ) + \mathds 1_{\{ y> \tau \}} \Lambda_0 ( u)\, : \, u\in \mathbb R_{\geq 0} \}$ is uniformly bounded and their subgraphs are ordered by inclusion, as $u $ increases. Therefore, any $2$ points can not be shattered by the collection of subgraphs, which means that the VC index is $2$. The class $\{x\mapsto  g(\gamma_0,x)\}$ has only one element, and hence the product will be Donsker as soon as $E[ g(\gamma_0,X)^2] <+\infty$ (again from Example 2.10.23 in \cite{vandervaart1996}).

As a second step, we show that the process  $\{n^{-1} \sum_{i=1}^n\int h(u) dM_i(u)\,:\, h\in \text{BV} (m,v)\}$ converges weakly in $\ell ^{\infty}( \text{BV} (m,v)) $ relying on the preservation of weak convergence through continuous mappings. The previous process is the image of $\{ n ^{-1} \sum_{i=1}^n M_i(u)\,:\, u\in \mathbb R_{\geq 0} \}$ by the linear transformation $H\mapsto \{\int h(u)dH(u)\, :\, h\in \text{BV} (m,v)\}$,  defined on the space of c\`ad-l\`ag functions and valued in $\ell ^{\infty}(\text{BV} (m,v))$. Weak convergence is preserved whenever the map is continuous \citep{vandervaart1996}, so whenever
\begin{align*}
\sup_{h\in \text{BV} (m,v)}\left|\int h(u) d H(u)\right| \rightarrow 0,\qquad \text{as } \sup_{u\in \mathbb R_{\geq 0}} |H(u)|\rightarrow 0.
\end{align*}
The latter holds since both norms are in fact equivalent \citep{dudley1992}.
\end{proof}

The following lemma is useful to characterize the limiting distribution of the estimators.
\begin{lemma}\label{lemma:weakcv}
Under (H\ref{cond:consistency_gamma_1}) and (H\ref{cond:asymptotics_gamma_2}), the empirical process
\begin{align*}
&    n^{-1/2} \sum_{i=1}^n \left( \begin{array}{c}
\int_0^y \frac{dM_i(u)}{Q_0(u) } \\
\int (d_0(X_i)- h_0(u)) dM_i(u)
\end{array}  \right)   ,
\end{align*}
converges weakly in $\ell^{\infty} (\R)\times \R^q$ to a Gaussian process with covariance function
\begin{align*}
(y,y') \mapsto \begin{pmatrix}
\int_0^{y\wedge y'}   \frac{d\Lambda_0(u)}{Q_0(u) }   &0\\
0& \int E\left[ (d_0(X)- h_0(u))(d_0(X)- h_0(u))^T g(\gamma_0,X) R(u)\right]   d\Lambda_0 (u)
\end{pmatrix} .
\end{align*} 
\end{lemma}
\begin{proof}
The statement is a consequence of Lemma \ref{Lemma:Donskerclass}. By (\ref{eq:bound_upper_lower}), $Q_0^{-1} \in \text{BV}(m,v)$ for some $m>0$ and $v>0$. Using that for any $j\in \{1,\ldots, q\}$, $u<u'$,
\begin{align*}
|\nabla_{\gamma,j} Q_{\gamma}(u) - \nabla_{\gamma,j} Q_{\gamma}(u')| \leq  E\left[  |\nabla_{\gamma,j} g(\gamma, X) | \, (R(u)-R(u'))   \right] ,
\end{align*}
we have that $\nabla_{\gamma,j} Q_{\gamma} \in \text{BV}(m',v')$ for some $m'>0$ and $v'>0$. Consequently, $h_0\in BV(m'',v'')$for some $m''>0$ and $v''>0$. The Donsker property given by Lemma \ref{Lemma:Donskerclass} implies the tightness of each coordinate of the underlying empirical process. Then by using the multivariate central limit theorem, we obtain the convergence in distribution of the finite dimensional laws. This shows the result.
\end{proof}

Recall that $\mathcal N_{}(\epsilon,\mathcal F, \|\cdot\|)$ denotes the $\epsilon$-covering number of the metric space $(\mathcal F,\|\cdot\|)$. A class $\mathcal F$ with envelop $F$ is said to satisfy the uniform entropy condition whenever
\begin{align}\label{cond:uniform_entropy}
\int_0^{+\infty} \sup_{Q} \sqrt {\log\mathcal N_{}\left(\epsilon\|F\|_{L_2(Q)},\mathcal F,L_2(Q)\right)   } d\epsilon <+\infty,
\end{align}
where the supremum is taken over all the finitely discrete probability measures. It is tempting to generalise the next Lemma with the Donsker property in place of the more technical and stronger requirement on the uniform entropy condition. Unfortunately it will generally fail when the random variables $g(\gamma,X)$, $\gamma\in B$, are unbounded. As detailed in \cite{vandervaart1996}, Section 2.10.2, stronger preservation properties are available when dealing with uniform entropy numbers \citep[Theorem 2.10.20]{vandervaart1996} rather than with the Donsker property \citep[Theorem 2.10.6]{vandervaart1996}.

\begin{lemma}\label{Lemma:equicontinuity1}
Let $\mathcal D$ denote a class of functions $\mathcal S\rightarrow \mathbb R$ satisfying (\ref{cond:uniform_entropy}) with envelop $D$ such that $E[D(X)^2g(\gamma_0,X) ]<+\infty$. If $\widehat d :\mathcal S \rightarrow \mathbb R$ is such that $\mathbb P(\widehat d\in \mathcal D) \rightarrow 1$ and $\int [\widehat d(X) ^2 g(\gamma_0,X)]dP_X =o_{\mathbb P}(1)$ (where $P_X$ is the probability measure of $X$), we have that
\begin{align*}
n^{-1} \sum_{i=1} ^n \int \widehat d (X_i)  dM_i(u) = o_{\mathbb P}(n^{-1/2}) .
\end{align*}
\end{lemma}

\begin{proof}
From the martingale property of $M$, each term of the previous empirical sum has mean $0$. To apply Theorem 2.1 in \cite{wellner2007}, two statements need to be verified. First, from Ito's isometry and because $ \int R(u) d\Lambda_0(u)\leq \theta_0$, we have
\begin{align*}
\int   \left(\int \widehat d (X) dM(u) \right)^2  dP_X&=\int  \left[   \int  \widehat d (X) ^2 g(\gamma_0, X) R(u) d\Lambda_0(u) \right] dP_X \\
&\leq \theta_0 \int \left[   \widehat d (X) ^2 g(\gamma_0, X)  \right] dP_X ,
\end{align*}
which goes to $0$ in $\mathbb P$-probability, by assumption. Second, the class of functions
\begin{align*}
\{ (y,\delta, x) \mapsto  d (x) \int dM_{y,\delta,x}(u) \, : \, d\in \mathcal D \}, 
\end{align*}
is shown to be Donsker by invoking Example 2.10.23 in \cite{vandervaart1996}. Indeed the class can be written as the product of two classes. One satisfies (\ref{cond:uniform_entropy}) and the other has only one element.  The condition on the envelop is $E[D(X)^2( \int dM)^2 ]\leq \theta_0 E[  D(X)^2 g(\gamma_0, X ) ]<+\infty   $ by assumption.
\end{proof}

\begin{lemma}\label{Lemma:equicontinuity2}
 Assume that there exist $v>0$ and $m>0$ such that $\mathbb P(\widehat h \in \text{BV} (m,v))\rightarrow 1$ and $h_0\in \text{BV} (m,v)$. If moreover, $\sup_{y\in \mathbb R_{\geq 0}} | \widehat h(y ) - h_0(y) | =o_{\mathbb P}(1)$, we have that
\begin{align*}
\sup_{y\in\mathbb R_{\geq 0}} \left| n^{-1} \sum_{i=1}^n \int_0^y (\widehat h (u) - h_0(u) ) d M_i (u)\right|  = o_{\mathbb P}(n^{-1/2}) .
\end{align*}
\end{lemma}

\begin{proof}
We rely on the asymptotic equicontinuity of empirical processes over Donsker classes. Denote by  $Z_n$ the process $\{ Z_n (h) = n^{-1/2} \sum_{i=1}^n \int  h (u) d M_i (u)\,:\, h\in \text{BV}(m,v)\}$.
From Lemma \ref{Lemma:Donskerclass}, $Z_n$ converges weakly in the space $\ell^{\infty}( \text{BV} (m,v))$. As asymptotic tightness is necessary to characterize weak convergence \cite[Theorem 1.5.7, see also page 89]{vandervaart1996}, for every $\epsilon>0$ and every $\eta>0$, there exists $\delta >0$ such that
\begin{align*}
\limsup_{n\rightarrow +\infty}\, \mathbb  P\left(\sup_{(h,\tilde h)\in \mathcal H _\delta  } |Z_n(h)- Z_n(\tilde h) | >\epsilon \right)< \eta ,
\end{align*}
where 
\begin{align*}
\mathcal H _\delta  = \left \{(h,\tilde h)\in \text{BV}(m,v)^2\, :\, \| h- \tilde h\|_{L_2(P)}\leq \delta\right   \}.
\end{align*}
We need to show that 
\begin{align*}
\sup_{y\in\mathbb R_{\geq 0}} \left|  Z_n(\widehat h\mathds 1_{[0, y]} )- Z_n(h_0\mathds 1_{[0, y]})\right| = o_{\mathbb P}(1).
\end{align*}
First, note that $\widehat h\mathds 1_{[0, y]}$ and $h_0\mathds 1_{[0, y]}$ belong to $\text{BV} (m,v)$, with probability going to $1$. Second, we have that
\begin{align*}
\sup_{y\in \mathbb R_{\geq 0}}\|\widehat h \mathds 1_{[0, y]} -h_0\mathds 1_{[0, y]}\|_{L_2(P)}\leq \sup_{y\in \mathbb R_{\geq 0}} | \widehat h(y ) - h_0(y) |,
\end{align*}
which goes to $0$ in $\mathbb P$-probability. Consequently, we have, with probability going to $1$, that
\begin{align*}
\left\{ (\widehat h\mathds 1_{[0, y]}, h_0\mathds 1_{[0, y]}) \, :\, y\in\mathbb R\right\}\subset \mathcal H_\delta ,
\end{align*}
which implies that
\begin{align*}
\sup_{y\in\mathbb R_{\geq 0}} \left|  Z_n(\widehat h\mathds 1_{[0, y]} )- Z_n(h_0\mathds 1_{[0, y]})\right| \leq \sup_{(h,\tilde h)\in \mathcal H _\delta  } |Z_n(h)- Z_n(\tilde h) | .
\end{align*}
The fact that $\epsilon $ and $\eta$ are arbitrarily small implies the statement.
\end{proof}

The application of Lemma \ref{Lemma:equicontinuity1} and Lemma \ref{Lemma:equicontinuity2} requires the following result, which establishes that $\w d_{\w \gamma}-d_0$ (resp. $\w h_{\w \gamma}$) verifies the conditions on $\w d$ (resp. $\w h$) in  Lemma \ref{Lemma:equicontinuity1} (resp.  Lemma \ref{Lemma:equicontinuity2}). In what follows, the $j$-th coordinate of $d_{\gamma} $, $\widehat h_{\gamma}$ and $\nabla _{\gamma} \widehat Q_\gamma $ is denoted by $d_{\gamma,j} $, $\widehat h_{\gamma,j}$, $\nabla _{\gamma,j} \widehat Q_\gamma  $, respectively ($j=1,\ldots, q)$.

\begin{lemma}\label{lemma:dgamma_donsker&hboundedvariation}
Under (H\ref{cond:consistency_gamma_0}), (H\ref{cond:consistency_gamma_1}) and (H\ref{cond:asymptotics_gamma_2}), for every $j\in \{1,\ldots, q\}$, the class of functions $\{x\mapsto d_{\gamma,j} (x) -d_{0,j} (x)  \,:\, \gamma\in B\} $ satisfies (\ref{cond:uniform_entropy}) with the envelop $L = \sqrt { 8}( \left({1}/{m_1 } \right) (\diam( B)c_2 +M_2)+  \left( {M_2}/{m_1^2} \right) (\diam( B)c_1 +M_1))$.
Moreover, there exists $m>0$ and $v>0$ such that $\mathbb P( \{ \w h_{\gamma,j} \, :\, \gamma\in B\}  \subset  \text{BV} (m,v) ) \rightarrow 1$.
\end{lemma}

\begin{proof}
The fact that \citep[Section 2.1.1]{vandervaart1996}
\begin{align*}
\mathcal N_{}\left(2 \epsilon  \|c_1\|_{L_2(Q)},\mathcal G, L_2(Q)\right)\leq \mathcal N_{[\,]}\left( 2 \epsilon  \|c_1\|_{L_2(Q)},\mathcal G, L_2(Q)\right),
\end{align*}
together with (\ref{ineq:bracketing_G}) when $p=2$, implies that
\begin{align}\label{ineq:covering_G}
\mathcal N_{}\left(2  \epsilon  \|c_1\|_{L_2(Q)},\mathcal G, L_2(Q)\right)\leq  K \epsilon^{-q}.
\end{align}
Let  $j\in \{1,\ldots, q\}$, and define  $\dot{\mathcal G}_j=\left\{ x \mapsto   \nabla _{\gamma,j} g(\gamma,x ) \,:\, \gamma\in B \right\}$. Then similarly as for the class $\mathcal G$, using (\ref{eq:mean_value_th2}) and invoking Theorem 2.7.11 in \cite{vandervaart1996} with the $L_2(Q)$-norm,  we find that
\begin{align}\label{ineq:covering_G_point}
\mathcal N_{}\left( 2\epsilon  \|c_2\|_{L_2(Q)},\dot{\mathcal G}_j, L_2(Q)\right) \leq K \epsilon^{-q}.
\end{align}
 The two previous inequalities continue to hold when the functions $2c_1$ and $2c_2$ are replaced by  $ \diam( B) c_1 +M_1  $ and $\diam( B)  c_2 +M_2  $, respectively. Because these two functions are enveloppes for $\mathcal G$ and $\dot{\mathcal G}_j$, (\ref{cond:uniform_entropy}) is satisfied for $\mathcal G$ and $\dot{\mathcal G}_j$ with the enveloppes $ L_1 = \diam( B) c_1 +M_1  $ and $L_2 = \diam( B)  c_2 +M_2  $. We are now interested in the quotient class formed by the elements $\dot g/g$, when $\dot g \in \dot{\mathcal G}_j$ and $g\in \mathcal G$. Note that, for every $\dot g_1, \dot g_2$ in $\dot{\mathcal G}_j$ and $g_1$, $g_2$ in $\mathcal G$,
\begin{align}
\left|\frac{ \dot g_1}{g_1} - \frac{\dot g_2}{g_1}\right|^2&\leq 2 \left|\frac{1}{\dot g_2 } \right| ^2\, |\dot g_1-\dot g_2|^2+2  \left| \frac{\dot g_2}{g _1g_2 } \right|^2\, |g_1-g_2|^2 \notag \\
&\leq   2\left|\frac{1}{m_1 } \right|^2 \, | \dot g_1-\dot g_2|^2 +  2 \left| \frac{M_2}{m_1^2} \right|^2 \, |g_1-g_2|^2 \label{ineq:dgamma_lipschitz}.
\end{align}
From the previous display, and because $\sqrt{ a+b} \leq \sqrt a+\sqrt b $ for $a\geq 0$ and $b\geq 0$,  an envelop for $ \dot{\mathcal G}_j / {\mathcal G} - d_{0,j}$ is given by $ \sqrt 8 ( (1/m_1) L_2 + (M_2/m_1^2) L_1 ) $ which is equal to $L$ given in the statement.
As (\ref{ineq:covering_G}), (\ref{ineq:covering_G_point}) and (\ref{ineq:dgamma_lipschitz}) holds, we can apply Theorem 2.10.20 in \cite{vandervaart1996} on the classes $\mathcal G$, $\dot{ \mathcal G}_j$ and $d_{ 0,k}$, to obtain that
\begin{align*}
\int_0^{+\infty} \sup_Q \sqrt { \log \mathcal N \left( \epsilon \| L \|_{L_2(Q)} , \dot{\mathcal G}_j / {\mathcal G}-d_{0,j}, L_2(Q)\right) } d\epsilon <+\infty ,
\end{align*}
where the supremum is taken over the finitely discrete probability measures. We have shown the first statement of the Lemma.

Let $\|f\| _{\text{tv}} $ denote the total variation of $f$ over $\overline {\mathbb R_{\geq 0}}$. To show the second statement, we need to prove that
\begin{align*}
\mathbb P\left(\sup_{\gamma\in B}\| \widehat h_{\gamma, j} \| _{\text{tv}}\leq v ,\, \sup_{\gamma\in B,\, y\in \mathbb R_{\geq 0}}| \widehat h_{\gamma, j}(y) |\leq m \right ) \longrightarrow 1.
\end{align*}
Define, for every $y\in \mathbb R_{\geq 0}$,
\begin{align*}
\widehat T_{\gamma,j}(y)= n^{-1} \sum_{i=1} ^n | \nabla _{\gamma,j}  g(\gamma, X_i ) | R_i(y).
\end{align*}
Introduce the event
\begin{eqnarray*}
A &=& \left \{ \omega :  E[ m_1(X) (1-\Delta)]/2 \leq  \widehat Q_\gamma (y )  \leq 2E[M_1(X)] \quad \right. \\
&& \hspace*{1cm} \left. \text{and } \widehat T_\gamma(u)\leq 2E[ M_2(X)] \text{ for all } \gamma\in B ,\,  y\in \mathbb R_{\geq 0} \right\} .
\end{eqnarray*}
On the set $A$, we have
\begin{align}\label{bound:1}
\sup_{\gamma\in B,\, y\in \mathbb R_{\geq 0}}\left| \widehat h_{\gamma, j}(y) \right|\leq   \frac{4E[M_2(X)] }{E[m_1(X)(1-\Delta)] }.
\end{align}
On $A$, we also have that, for all $\gamma\in B$ and $u < v$ in $\overline{\mathbb R_{\geq 0}}$,
\begin{align*}
|\widehat h_{\gamma, j} (u) -\widehat h_{\gamma, j} (v) | &= \left| \frac{\nabla _{\gamma,j} \widehat Q_\gamma (u) }{\widehat Q_\gamma (u)} - \frac{\nabla _{\gamma,j} \widehat Q_\gamma (v) }{\widehat Q_\gamma (v)} \right| \\
&\leq \left( 2 \frac{ | \nabla _{\gamma,j} \widehat Q_\gamma (u)-\nabla _{\gamma,j} \widehat Q_\gamma (v)|}{E[m_1(X)(1-\Delta)] } + 8\frac{E[M_2(X)] \, | \widehat Q_\gamma (u) -\widehat Q_\gamma (v) | }{E[m_1(X)(1-\Delta)] ^{2}}\right) .
\end{align*} 
It follows that, there exists a $C>0$ such that, for all $\gamma\in B$ and $u < v$ in $\overline{\mathbb R_{\geq 0}}$,
\begin{align*}
| \widehat h_{\gamma, j} (u) - \widehat h_{\gamma, j} (v) | \leq C \left(  | \widehat T_{\gamma,j} (u)- \widehat T_{\gamma ,j} (v)| + | \widehat Q_\gamma (u) -\widehat Q_\gamma (v) |\right).
\end{align*} 
Apply the previous inequality and use the fact that $\widehat T_\gamma$ and $\widehat Q_\gamma$ are non-increasing functions to obtain that, on $A$, for all $\gamma\in B$ and any set of points $u_0<u_1\ldots <u_N $,
\begin{align*}
\sum_{k=1}^N | \widehat h_{\gamma, j} (u_k) - \widehat h_{\gamma, j} (u_{k-1}) | &\leq C  \left(   \sum_{k=1}^N | \widehat T_{\gamma,j} (u_k)- \widehat T_{\gamma,j} (u_{k-1})| +\sum_{k=1}^N | \widehat Q_\gamma (u_k) -\widehat Q_\gamma (u_{k-1}) |\right)\\
&=  C  \left ( \widehat T_{\gamma,j} (u_0)- \widehat T_{\gamma,j} (u_{N}) +\widehat Q_\gamma (u_0) -\widehat Q_\gamma (u_{N}) \right)\\
&\leq C   \left( \widehat T_{\gamma,j} (0)  +\widehat Q_\gamma (0) \right)\\
&\leq  C   \left( \sup_{\gamma\in B} \{ \widehat T_{\gamma,j} (0) \}   +\sup_{\gamma\in B} \{ \widehat Q_\gamma (0)\} \right).
\end{align*} 
Consequently, on $A$,
 \begin{align}\label{bound:2}
\sup_{\gamma\in B} \| \widehat h_{\gamma, k}  \| _{\text{tv}} \leq  2 C(E[M_1(X)]+E[M_2(X)]).  
 \end{align}
Hence, with (\ref{bound:1}) and (\ref{bound:2}) we have found $m$ and $v$ such that
\begin{align*}
\mathbb P(A) \leq \mathbb P\left(\sup_{\gamma\in B}\| \widehat h_{\gamma, j} \| _{\text{tv}}\leq v ,\, \sup_{\gamma\in B,\, y\in \mathbb R_{\geq 0}}| \widehat h_{\gamma, j}(y) |\leq m \right ).
\end{align*} 
From Lemma \ref{Lemma:probacv}, statements (\ref{convergence:12}) and (\ref{convergence:21.1}), $\mathbb P(A) $ goes to $1$ and hence the result follows.
\end{proof}

\end{appendices}

\vspace*{.5cm}

\noindent
{\bf \large Acknowledgments} \\
This work was supported by Interuniversity Attraction Pole Research Network P7/06 of the Belgian State (Belgian Science Policy).  F. Portier was in addition supported by Fonds de la Recherche Scientifique (FNRS) A4/5 FC 2779/2014- 2017 No. 22342320, I. Van Keilegom was also supported by the European Research Council (2016-2021, Horizon 2020/ ERC grant agreement No.\ 694409), and A. El Ghouch was also supported by the PDR (convention PDR.T.0080.16), a funding instrument of the FNRS.

\bibliographystyle{chicago}
\bibliography{finalversion_arkiv.bbl}

\end{document}